\documentclass[letterpaper,11pt]{article}
\usepackage[toc,page]{appendix}
\usepackage[margin=1in]{geometry}
\usepackage[bookmarks, colorlinks=true, plainpages = false, citecolor = blue,linkcolor=red,urlcolor = blue, filecolor = blue,pagebackref]{hyperref}
\usepackage{amsmath,amsfonts,amsthm,amssymb,bm, verbatim,dsfont,mathtools}
\usepackage{color,graphicx,appendix}
\usepackage{subfigure}
\usepackage{etoolbox}
\usepackage{tikz}
\usepackage{xr,xspace}
\usepackage{todonotes}
\usepackage{paralist}
\usepackage{caption,soul}
\usepackage{algorithm}
\usepackage{algorithmic}
\usepackage{enumitem}
\makeatletter

\newtheorem{theorem}{Theorem}
\newtheorem{lemma}{Lemma}
\newtheorem{proposition}{Proposition}

\theoremstyle{definition}
\newtheorem{definition}{Definition}
\newtheorem{hypothesis}{Hypothesis}

\newtheorem{remark}{Remark}

\usepackage{tikz}

\usepackage{xspace,prettyref}
\usepackage{bm}
\newrefformat{eq}{(\ref{#1})}
\newrefformat{chap}{Chapter~\ref{#1}}
\newrefformat{sec}{Section~\ref{#1}}
\newrefformat{alg}{Algorithm~\ref{#1}}
\newrefformat{fig}{Fig.~\ref{#1}}
\newrefformat{tab}{Table~\ref{#1}}
\newrefformat{rmk}{Remark~\ref{#1}}
\newrefformat{clm}{Claim~\ref{#1}}
\newrefformat{def}{Definition~\ref{#1}}
\newrefformat{cor}{Corollary~\ref{#1}}
\newrefformat{lmm}{Lemma~\ref{#1}}
\newrefformat{prop}{Proposition~\ref{#1}}
\newrefformat{app}{Appendix~\ref{#1}}
\newrefformat{hyp}{Hypothesis~\ref{#1}}
\newrefformat{thm}{Theorem~\ref{#1}}
\newrefformat{ass}{Assumption~\ref{#1}}
\newrefformat{conj}{Conjecture~\ref{#1}}

\newcommand{\ie}{i.e.\xspace}
\renewcommand{\P}{\mathcal{P} }
\newcommand{\Q}{\mathcal{Q}}
\newcommand{\Exp}{\mathbb{E}}

\newcommand{\indicator}[1]{\bm{1}_{#1}}
\renewcommand{\hat}{\widehat}
\renewcommand{\tilde}{\widetilde}
\newcommand{\argmax}{\mathrm{argmax}}

\usepackage{color}

\newcommand{\blue}{\color{blue}}
\newcommand{\nb}[1]{{\sf\blue[#1]}}

\renewcommand{\implies}{\Rightarrow}
\newcommand{\1}[1]{{\mathbf{1}_{\left\{{#1}\right\}}}}

\newcommand{\Hyper}{\text{Hypergeometric}}

\newcommand{\TV}{d_{\rm TV}}

\usepackage{xspace,prettyref}
\newcommand{\CML}{\widehat{C}_{\rm ML}}
\newcommand{\diverge}{\to\infty}

\newcommand{\iiddistr}{{\stackrel{\text{\iid}}{\sim}}}
\newcommand{\ones}{\mathbf 1}

\newcommand{\reals}{{\mathbb{R}}}

\newcommand{\naturals}{{\mathbb{N}}}

\newcommand{\eexp}{{\rm e}}

\newcommand{\identity}{\mathbf I}
\newcommand{\allones}{\mathbf J}

\newcommand{\diff}{{\rm d}}

\newcommand{\Expect}{\mathbb{E}}
\newcommand{\expect}[1]{\mathbb{E}\left[ #1 \right]}
\newcommand{\eexpect}[1]{\mathbb{E}[ #1 ]}

\newcommand{\Prob}{\mathbb{P}}
\newcommand{\pprob}[1]{ \mathbb{P}\{ #1 \} }
\newcommand{\prob}[1]{ \mathbb{P}\left\{ #1 \right\} }

\newcommand{\Bern}{{\rm Bern}}
\newcommand{\Binom}{{\rm Binom}}

\newcommand{\eg}{e.g.\xspace}
\newcommand{\iid}{i.i.d.\xspace}
\newrefformat{eq}{(\ref{#1})}
\newrefformat{chap}{Chapter~\ref{#1}}
\newrefformat{sec}{Section~\ref{#1}}
\newrefformat{alg}{Algorithm~\ref{#1}}
\newrefformat{fig}{Fig.~\ref{#1}}
\newrefformat{tab}{Table~\ref{#1}}
\newrefformat{rmk}{Remark~\ref{#1}}
\newrefformat{clm}{Claim~\ref{#1}}
\newrefformat{def}{Definition~\ref{#1}}
\newrefformat{cor}{Corollary~\ref{#1}}
\newrefformat{lmm}{Lemma~\ref{#1}}
\newrefformat{prop}{Proposition~\ref{#1}}
\newrefformat{app}{Appendix~\ref{#1}}
\newrefformat{hyp}{Hypothesis~\ref{#1}}
\newrefformat{thm}{Theorem~\ref{#1}}
\newrefformat{ass}{Assumption~\ref{#1}}

\newcommand{\pth}[1]{\left( #1 \right)}
\newcommand{\qth}[1]{\left[ #1 \right]}

\newcommand{\iprod}[2]{\left \langle #1, #2 \right\rangle}
\newcommand{\Iprod}[2]{\langle #1, #2 \rangle}
\newcommand{\indc}[1]{{\mathbf{1}_{\left\{{#1}\right\}}}}

\newcommand{\tz}{{\widetilde{z}}}
\newcommand{\tA}{{\widetilde{A}}}

\newcommand{\tG}{{\widetilde{G}}}

\newcommand{\tP}{{\widetilde{P}}}

\newcommand{\tV}{{\widetilde{V}}}

\newcommand{\calC}{{\mathcal{C}}}

\newcommand{\calE}{{\mathcal{E}}}

\newcommand{\calG}{{\mathcal{G}}}

\newcommand{\calN}{{\mathcal{N}}}

\newcommand{\calP}{{\mathcal{P}}}

\newcommand{\calS}{{\mathcal{S}}}
\newcommand{\calT}{{\mathcal{T}}}

\newcommand{\Th}{{\rm th}}

\newcommand{\PDS}{{\sf PDS}\xspace}
\newcommand{\PC}{{\sf PC}\xspace}

\newcommand{\ER}{Erd\H{o}s-R\'enyi\xspace}

\renewcommand{\hat}{\widehat}
\renewcommand{\tilde}{\widetilde}

\newcommand{\MMSE}{{\rm MMSE}}

\newcommand{\tsigma}{\tilde{\sigma}}

\newcommand{\score}{\calT}

\begin{document}

\title{Statistical Problems with Planted Structures: Information-Theoretical and Computational Limits}

\date{\today}
\author{ 
Yihong Wu \and Jiaming Xu\thanks{
Y.~Wu is with Department of Statistics and Data Science, 
Yale University, New Haven, CT 06520, USA, \texttt{yihong.wu@yale.edu}.
J.~Xu is with the Fuqua School of Business, Duke University,
Durham, NC 27708, \texttt{jiaming.xu868@duke.edu}.}
}
  
  \maketitle

\begin{abstract}
Over the past few years, insights from computer science, statistical physics, and information theory have revealed phase transitions in a wide array of high-dimensional statistical problems at two distinct thresholds: One is the information-theoretical (IT) threshold below which the observation is too noisy so that inference of the ground truth structure is impossible regardless of the computational cost; the other is the computational threshold above which inference can be performed efficiently, i.e., in time that is polynomial in the input size. In the intermediate regime, inference is information-theoretically possible, but conjectured to be computationally hard.

This article provides a survey of the common techniques for determining the sharp IT and computational limits, using community detection and submatrix detection as illustrating examples. For IT limits, we discuss tools including the first and second moment method for analyzing the maximum likelihood estimator, information-theoretic methods for proving impossibility results using  mutual information and rate-distortion theory, and methods originated from statistical physics such as interpolation method. To investigate computational limits, we describe a common recipe to construct a randomized polynomial-time reduction scheme that approximately maps instances of the planted clique problem to the problem of interest in total variation distance.
\end{abstract}

\tableofcontents
  
 \section{Introduction}
\label{sec:intro}


 %

The interplay between information theory and statistics is a constant theme in the development of both fields.
Since its inception, information theory has been indispensable for understanding the fundamental limits of statistical inference.
Classical \emph{information bound} provides fundamental lower bounds for the estimation error, including Cram\'er-Rao and Hammersley-Chapman-Robbins lower bounds in terms of Fisher information and  $\chi^2$-divergence \cite{lehmann,BL91}. 
In the classical ``large-sample'' regime in parametric statistics, Fisher information also governs the sharp minimax risk in regular statistical models \cite{VdV00}. The prominent role of information-theoretic quantities such as mutual information, metric entropy and capacity in establishing the minimax rates of estimation has long been recognized since the seminal work of \cite{Lecam86,IKbook,Birge83,YB99}, etc. 

Instead of focusing on the large-sample asymptotics, the attention of contemporary statistics has shifted toward \emph{high dimensions}, where both the problem size and the sample size grow simultaneously and the main objective is to obtain a tight characterization of the optimal statistical risk. Information-theoretic methods have been remarkably successful for high-dimensional problems, such as methods based on metric entropy and Fano's inequality for determining the minimax risk  within universal constant factors (minimax rates) \cite{YB99}. Unfortunately,  the aforementioned methods are often too crude for the task of determining the \emph{sharp constant}, which requires more refined analysis and stronger information-theoretic tools.

An additional challenge in dealing with high dimensionality is to address the computational aspect of statistical inference. An important element absent from the classical statistical paradigm is the computational complexity of inference procedures, which is becoming increasingly relevant for data scientists dealing with large-scale noisy datasets. Indeed, recent results \cite{berthet2013lowerSparsePCA,ma2013submatrix,HajekWuXu14,wang2016statistical,gao2017sparse,Brennan18}
 revealed the surprising phenomenon that certain problems concerning large networks and matrices undergoes an ``easy-hard-impossible'' phase transition and computational constraints can severely penalize the statistical performance.
It is worth pointing out that here the notion of complexity differs from the worst-case computational hardness studied in the computer science literature which focused on the time and space complexity of various worst-case problems. 
In contrast, in a statistical context, the problem is of a stochastic nature and the existing theory on average-case hardness is significantly underdeveloped. 
Here, the hardness of a statistical problem is often established either under certain computation models, such as the sums-of-squares 
relaxation hierarchy, or by means of a reduction argument from another problem, notably the planted clique problem, which is conjectured to be computationally intractable.  



In this article we provide an exposition on some of the methods for determining the information-theoretic (IT) as well as computational limits for high-dimensional statistical problems with a planted structure, with specific focus on characterizing sharp thresholds. 
Here the planted structure refers to the true parameter, referred to as the ground truth, which is often of a combinatorial nature (e.g. partition) and hidden in the presence of random noise. To characterize the IT limit, we will discuss tools including the first and second moment method for analyzing the maximum likelihood estimator, information-theoretic methods for proving impossibility results using  mutual information and rate-distortion theory, and methods originated from statistical physics such as the interpolation method.
Determining the computational limit of statistical problems, especially the ``easy-hard-impossible'' phase transition, has no established recipe and is usually done on case-by-case basis; nevertheless, the common element is to construct a randomized polynomial-time reduction scheme that \emph{approximately} maps instances of a given hard problem to one that is close to the problem of interest in total variation distance.

 \section{Basic Setup}

%
%
%
%

To be concrete, in this article we consider two representative problems, namely, \emph{community detection}  and \emph{submatrix detection}
as running examples. 
Both problems can be cast as the Bernoulli and Gaussian version of the following statistical model with planted community structure.

	
We first consider a random graph model containing a single hidden community whose size can be sublinear in data matrix size $n$.

\begin{definition}[Single Community Model]  \label{def:single}
 Let $C^*$ be drawn uniformly at random from all subsets of $[n]$ of cardinality $K$.
 Given probability measures $P$ and $Q$ on a common measurable space, 
 let $A$ be an $n \times n$ symmetric matrix with zero diagonal
 where for all $1 \le i<j \le n$, $A_{ij}$ are mutually independent, and $A_{ij} \sim P$ if $i,j\in C^*$ and $A_{ij} \sim Q$ otherwise.
\end{definition}

In this paper we assume that we only have access to pairwise information $A_{ij}$ for distinct indices $i$ and $j$ whose distribution is either $P$ or $Q$ depending on the community membership;
no direct observation about the individual indices is available (hence the zero diagonal of $A$).
Two choices of $P$ and $Q$ arising in many applications are the following:
\begin{itemize}
\item {Bernoulli case}: 
 $P=\Bern(p)$ and $Q=\Bern(q)$ with $p\neq q $. When $p>q$, this coincides with the {\em planted dense subgraph model}
studied in \cite{McSherry01,arias2013community,ChenXu14,HajekWuXu14,Montanari:15OneComm},
which is also a special case of the general stochastic block model (SBM)~\cite{Holland83} with a single community.\index{Stochastic block model (SBM)}\index{Stochastic block model (SBM)! Single community}
In this case, the data matrix $A$ corresponds to the adjacency matrix of a graph, where two vertices are connected with probability $p$ if both belong to the community $C^\ast$, and with probability $q$ otherwise.
Since $p > q$, the subgraph induced by $C^\ast$ is likely to be denser than the rest
of the graph. 

\item {Gaussian case}: $P=\calN(\mu,1)$ and $Q=\calN(0,1)$ with $\mu \neq 0$. 
This corresponds to a symmetric version of the {\em submatrix detection} problem studied in \cite{shabalin2009submatrix,kolar2011submatrix,butucea2013,Butucea2013sharp,ma2013submatrix,ChenXu14,CLR15}.\index{Submatrix detection problem}
When $\mu>0$, the entries of $A$ with row and column indices in $C^\ast$ have positive mean $\mu$ except those on the diagonal, while the rest of the entries have zero mean. 
\end{itemize}

We will also consider a binary symmetric community model with two communities of equal sizes. 
The Bernoulli case is known as binary symmetric stochastic block model (SBM).\index{Binary symmetric community model}\index{Stochastic block model (SBM)! Two communities}

 \begin{definition}[Binary symmetric community model]
\label{def:binary}
 Let $(C_1^*,C_2^*)$ be two communities of equal size that are drawn uniformly at random from all equal-sized partitions of $[n]$.
 Let $A$ be an $n \times n$ symmetric matrix with empty diagonal
 where for all $1 \le i<j \le n$, $A_{ij}$ are mutually independent, and $A_{ij} \sim P$ if $i,j$ are from the same
 community and $A_{ij} \sim Q$ otherwise.
 \end{definition}

Given the data matrix $A$, the problem of interest is to accurately recover the underlying single community $C^*$ or community
partition $(C_1^*,C_2^*)$ up to a permutation of cluster indices. The distributions $P$ and $Q$ as well as the community size $K$
depend on the matrix size $n$ in general. For simplicity
we assume that these model parameters are known to
the estimator. Common objectives of recovery include the following:
\begin{itemize}
\item \textbf{Detection}: detect the presence of planted communities versus the absence. This is a hypothesis problem: in the null case the observation consists of purely noise with independently and identically distributed (iid) entries, 
while in the alternative case the distribution of the entries are dependent on the hidden communities per \prettyref{def:single} or \ref{def:binary}.
	\item \textbf{Correlated recovery}: recover the hidden communities better than random guessing. 
	For example, for the binary symmetric SBM, the goal is to achieve a misclassification rate strictly less than $1/2$.
	\item \textbf{Almost exact recovery}: The expected number of misclassified vertices is sublinear in the hidden community sizes. 
	\item \textbf{Exact recovery}: All vertices are classified correctly with probability converging to $1$ as the dimension $n \to \infty.$ 
\end{itemize}


 \section{Information-theoretic limits}

\subsection{Detection and correlated recovery}
	\label{sec:detection}
In this subsection, we study  detection and correlated recovery under the binary symmetric community model.
The community structure under the binary symmetric community model can be represented by a vector $\sigma \in \{\pm 1\}^n$
such that $\sigma_i=1$ if vertex $i$ is in the first community and $\sigma_i=-1$ otherwise. 
Let $\sigma^*$ denote the true community partition and 
$\hat{\sigma} \in \{\pm 1\}^n$ denote an estimator of $\sigma$. 
For detection, we assume under the null model, $A_{ii}=0$ for all $1 \le i \le n$ and $A_{ij}=A_{ij}$ are iid as $\frac{1}{2}(P+Q)$ for $1 \le i < j \le n$, so that $\expect{A}$ is matched between the planted and null model.

\begin{definition}[Detection]
Let $\P$ be the distribution of $A$ in the planted model, 
and denote by $\Q$ the distribution of $A$ in the null model.  A test statistic $\calT(A)$ with a threshold $\tau$ achieves detection if\footnote{This criterion is also known as strong detection, 
in contrast to weak detection which
only requires 
$\P( \calT(A) < \tau ) + \Q( \calT(A) \ge \tau )$ to be bounded away from $1$
as $n \to \infty$. In this paper, we focus exclusively on strong detection. 
See~\cite{PerryWeinBandeira16,alaoui2017finite} for detailed discussions on weak detection.
}
\[
\limsup_{n \to \infty} \left[ \P( \calT(A) < \tau) + \Q(\calT(A) \ge \tau) \right] = 0,
\]
 so that the criterion $\calT(A) \ge \tau$ determines with high probability whether $A$ is drawn from $\P$ or $\Q$.\index{Detection}
\end{definition}

\begin{definition}[Correlated Recovery]
Estimator $\hat{\sigma}$ achieves correlated recovery of $\sigma^*$ if there exists a fixed
constant $\epsilon>0$ such that  $\eexpect{\left| \iprod{\sigma}{\sigma^*} \right|} \ge \epsilon n$ for all $n$.\index{Correlated Recovery}
\end{definition}

The detection problem can be understood as a binary hypothesis testing problem.  Given a test statistic $\calT(A)$, we consider its distribution under the planted and null models.  
If these two distributions are asymptotically disjoint, i.e., their total variation distance tends to $1$ in the limit of large datasets, then it is information-theoretically possible to distinguish the two models with high probability by measuring $\calT(A).$ A classic choice of statistic for binary hypothesis testing
is the likelihood ratio, 
\[
	\frac{\P(A)}{\Q(A)} 
	= \frac{\sum_\sigma \P(A, \sigma)}{\Q(A)} 
	= \frac{\sum_\sigma \P(A | \sigma) \,\P(\sigma)}{\Q(A)} \, . 
\]
This object will figure heavily in both our upper and lower bounds of the detection threshold. 

Before presenting our proof techniques, we first give the sharp threshold for 
detection and correlated recovery under the binary symmetric community model.

\begin{theorem}
\label{thm:corr}
Consider the binary symmetric community model. 
\begin{itemize}
\item If  $P=\Bern(a/n)$ and $Q=\Bern(b/n)$ for 
fixed constants $a, b$, then both detection and correlated recovery are information-theoretically 
possible  when $(a-b)^2>2(a+b)$ and impossible when $(a-b)^2<2(a+b)$.
\item If $P=\calN(\mu/\sqrt{n},1)$ and $Q=\calN(- \mu/\sqrt{n},1)$, then both detection and correlated recovery are information-theoretically possible when $\mu>1$
and impossible when $\mu<1$. 
\end{itemize}
\end{theorem}

We will explain how to prove the converse part of \prettyref{thm:corr}  using  second moment analysis of 
the likelihood ratio $\P(A)/\Q(A)$ and mutual information arguments. 
For the positive part of \prettyref{thm:corr}, we will present a simple first moment 
method to derive upper bounds that often coincide with the sharp thresholds up to a multiplicative constant. 

To achieve the sharp detection upper bound for the SBM, one can use the count of 
short cycles as test statistics as in \cite{Mossel13}. To achieve the sharp detection threshold 
in the Gaussian model and correlated recovery threshold in both models, one can resort to spectral methods.
For the Gaussian case, this directly follows from a celebrated phase transition result on the rank-one perturbation of Wigner matrices~\cite{baik2005phase,peche2006largest,benaych2011eigenvalues}. For the SBM, naive spectral methods fail due to the existence of high-degree vertices~\cite{KrzakalaMMSLC13spectral}. 
More sophisticated spectral methods based on self-avoiding walks or
non-backtracking walks have been shown to achieve the sharp correlated recovery threshold efficiently \cite{Massoulie13,Mossel13,BordenaveLelargeMassoulie:2015dq}.

\subsubsection{First moment method for detection and correlated recovery upper bound}\index{First moment method}
Our upper bounds do not use the likelihood ratio directly, since it is hard to furnish lower bounds on the typical value of $\P(A)/\Q(A)$ when $A$ is drawn from $\P$.
  Instead, we use the generalized likelihood ratio
\[
	\max_\sigma \frac{\P(A | \sigma)}{\Q(A)} \, 
\]
as the test statistic.
In the planted model where the underlying true community is $\sigma^*$, this quantity is trivially bounded below by $\P(A | \sigma^* )/\Q(A)$.  Then using a simple first moment argument (union bound) one can show that, in the null model $\Q$, with high probability, this lower bound is not achieved by any $\sigma$ and hence the generalized likelihood ratio test succeeds.  
An easy extension of this argument shows that, in the planted model, the maximum likelihood estimator (MLE)\index{Maximum Likelihood Estimator (MLE)}
$$
\widehat{\sigma}_{\rm ML} = \argmax_\sigma \P(A | \sigma)
$$ 
has nonzero correlation with $\sigma^*$, achieving the correlated recovery.  

Note that the first moment analysis of MLE often falls short of proving the sharp detection and correlated recovery upper bound. For instance, as we will explain next, the first moment calculation in~\cite[Theorem 2]{Banks16}
only shows the MLE achieves detection and correlated recovery when $\mu>2\sqrt{\log 2}$ in the Gaussian model,\footnote{Throughout this article, logarithmic is with respect to the natural base.} which is suboptimal in view of \prettyref{thm:corr}.
One reason is that the naive union bound in the first moment analysis may not
be tight; it does not take into the account the correlation between $P(A|\sigma)$ and $P(A|\sigma')$ for two different
$\sigma, \sigma'$ under the null model.

Next we explain how to carry out the first moment analysis in 
the Gaussian case with $P=\calN(\mu/\sqrt{n},1)$
and $Q=\calN( -\mu/\sqrt{n},1)$.
Specifically, assume $A= (\mu/\sqrt n)\left( \sigma^* (\sigma^*)^\top - \identity \right) +W$,
where  $W$ is a symmetric Gaussian random variable with zero diagonal and
$W_{ij} \iiddistr \calN(0,1)$ for $i<j$. It follows that 
$
	\log \frac{\P(A | \sigma)}{\Q(A)} = \frac{\mu}{\sqrt{n}} \sum_{i<j} A_{ij} \sigma_i \sigma_j + \frac{\mu^2 (n-1)}{4}.
$
Therefore, the generalized likelihood test reduces to test statistic 
$
	\max_{\sigma} \score(\sigma) \triangleq \sum_{i < j} A_{i,j} \sigma_i \sigma_j. 
$
Under the null model $\Q$, $
	\score(\sigma) \sim \calN\left(0, \frac{n (n-1) }{2} \right).$
Under the planted model $\P$, $
	\score(\sigma)  = \frac{\mu}{\sqrt{n} } \sum_{i<j} \sigma_i^* \sigma_j^* \sigma_i \sigma_j 
	 + \sum_{i < j} W_{i,j} \sigma_i \sigma_j.$
Hence the distribution of $\score(\sigma)$ depends on the overlap  $|\Iprod{\sigma}{\sigma^*}|$ 
between $\sigma$ and the planted partition $\sigma^*$. Suppose $|\Iprod{\sigma}{\sigma^*}|= n\omega$. Then 
\begin{align*}
	\score(\sigma) \sim \calN\left(\frac{\mu n ( n\omega^2-1) }{2 \sqrt n}  , \frac{n (n-1) }{2} \right).
\end{align*}
To prove that detection is possible, notice that in the planted model, 
$\max_\sigma \score(\sigma) \ge \score(\sigma^*)$.
Setting $\omega = 1$, Gaussian tail bounds yield that
\begin{align*}
	\P\left[\score(\sigma^*) \le \frac{\mu n(n-1) }{2 \sqrt n} -  n \sqrt{\log n} \right] \le n^{-1}.
\end{align*}
Under the null model,  taking the union bound over at most $2^n$ ways to choose $\sigma$, 
we can bound the probability that \emph{any} partition is as good, according to $\score$, as the planted one, by
\begin{align*}
	 \Q\left[\max_{\sigma}\score(\sigma) > \frac{\mu n(n-1) }{2 \sqrt n} -  n \sqrt{\log n} \right]   \le 
	 2^n \exp \left( 
	 - n \left( \frac{\mu}{2} \sqrt{\frac{n-1}{n}} - 
	\sqrt{\frac{\log n}{n-1} } \right)^2  \right) .
\end{align*}
Thus the probability of this event is $e^{-\Omega(n)}$ whenever $ \mu > 2 \sqrt{\log 2}$, 
meaning that above this threshold we can distinguish the null and planted models with 
the generalized likelihood test.

To prove that correlated recovery is possible, 
since $\mu > 2 \sqrt{\log 2}$, there exists a fixed $\epsilon>0$ such that $\mu (1-\epsilon^2) > 2  
\sqrt{\log 2}$. 
Taking the union bound over every $\sigma$ with $|\Iprod{\sigma}{\sigma^*}| \le n \epsilon$ gives
\begin{align*}
	 &~\P\left[\max_{ |\Iprod{\sigma}{\sigma^*}| \le n \epsilon }\score(\sigma) 
	\ge \frac{\mu n (n-1)  }{2\sqrt n} - n \sqrt{\log n} \right] \\
	\le &~2^n \exp \left( -  n\left( \frac{\mu (1-\epsilon^2) }{2} \sqrt{\frac{n}{n-1}} - 
	\sqrt{\frac{\log n}{n-1} } \right)^2  \right).
\end{align*}
 Hence, with probability at least $1-e^{-\Omega(n)}$, 
$$
\max_{ |\Iprod{\sigma}{\sigma^*}| \le n \epsilon } \score(\sigma)  <  \frac{n (n-1) \mu }{2\sqrt n} - n \sqrt{\log n},
$$
and consequently $ | \Iprod{ \hat{\sigma}_{\rm ML} }{ \sigma^*} |  \ge n \epsilon $ with high probability.
Thus, $ \hat{\sigma}_{\rm ML}$ achieves correlated recovery.

\subsubsection{Second moment method for detection lower bound}\index{Second moment method}
Intuitively, if the planted model $\P$ and the null model $\Q$ are close to being mutually singular, then the likelihood ratio $\P/\Q$ is almost always either very large or close to zero.  In particular, its variance under $\Q$, that is, the $\chi^2$-divergence\index{$\chi^2$-divergence}
\[
\chi^2(\P\|\Q) \triangleq  \Exp_{A \sim \Q} 
 \left[ \left( \frac{\P(A) } {\Q(A)} -1 \right)^{2} \right] 
\]
 must diverge. This suggests that we can derive lower bounds on the detection threshold 
by bounding the second moment of $\P(A)/\Q(A)$ under $\Q$, or equivalently its expectation under $\P$.  
Suppose the second moment is bounded by some constant $C$, \ie,
\begin{align}
\label{eq:smm-bound}
\chi^2(\P\|\Q)+1=
\Exp_{A \sim \Q} 
 \left[ \left( \frac{\P(A) } {\Q(A)} \right)^{2} \right] 
= \Exp_{A \sim \P}  \left[ \frac{\P(A)}{\Q(A)} \right] 
\le C \, .
\end{align}
A bounded second moment readily implies a bounded Kullback-Leibler divergence between $\P$ and $\Q$, since Jensen's inequality gives
 \begin{align}
 D (\P \| \Q)  = \Exp_{A \sim \P} \log \frac{\P(A)}{\Q(A)} \le \log \Exp_{A \sim \P} \frac{\P(A)}{\Q(A)} 
 \le \log C = O(1) \, . \label{eq:secondmoment_KL}
 \end{align}
 Moreover, it also implies non-detectability. To see this, 
 let $E=E_n$ be a sequence of events such that $\Q(E) \to 0$ as $n\to \infty$, 
 and let $\indicator{E}$ denote the indicator random variable for $E$.  Then the Cauchy-Schwarz inequality gives
 \begin{align} 
	\P(E) 
	= \Exp_{A \sim \Q} \frac{\P(A)}{\Q(A)} \,\indicator{E}  
	 \le \sqrt{
	\Exp_{A \sim \Q}  \left( \frac{\P(A)}{\Q(A)} \right)^{\!2}  
	\times \Exp_{A \sim \Q} \,\indicator{E}^2 }  
	 \le \sqrt{C \Q(E)} \to 0 \, . \label{eq: contiguity}
\end{align}
In other words, the sequence of distributions $\P$ is \emph{contiguous} to $\Q$ \cite{Lecam86}.\index{Contiguity}
Therefore, 
no algorithm can return ``yes'' with high probability (or even positive probability) in the planted model, and ``no'' with high probability in the null model. Hence, detection is impossible.

Next we explain how to compute the $\chi^2$-divergence for the binary symmetric SBM. 
One useful observation due to \cite{IS03} (see also \cite[Lemma 21.1]{it-stats}) is that, using Fubini's theorem, the $\chi^2$-divergence between a mixture distribution and a simple distribution can be written as
\[
\chi^2(\P\|\Q) + 1 = \Expect_{\sigma,\tsigma}\qth{ \Expect_{A \sim \Q}\qth{\frac{\P(A|\sigma)\P(A|\tilde\sigma)}{\Q(A)}}},
\]
where $\tilde\sigma$ is an independent copy of $\sigma$. 
Note that under the planted model $\calP$, the distribution of the $ij$th entry is given by 
$P \indc{\sigma_i=\sigma_j} + Q \indc{\sigma_i \neq \sigma_j} = \frac{P+Q}{2} + \frac{P-Q}{2} \sigma_i \sigma_j$.
Thus\footnote{In fact, the quantity $\rho=\int\frac{(P-Q)^2}{2(P+Q)}$ is an $f$-divergence known as the Vincze-Le Cam distance \cite{Lecam86,Vajda09}.}
\begin{align}
\chi^2(\P\|\Q) + 1
= & ~ \Expect\qth{\prod_{i < j} \int \frac{(\frac{P+Q}{2} + \frac{P-Q}{2} \sigma_i \sigma_j)(\frac{P+Q}{2} + \frac{P-Q}{2} \tsigma_i \tsigma_j)}{\frac{P+Q}{2}}} \nonumber \\
= & ~ \Expect\Bigg[\prod_{i < j} \Bigg(1 +  \sigma_i \sigma_j \tsigma_i \tsigma_j \underbrace{\int \frac{(P-Q)^2}{2(P+Q)}}_{\triangleq \rho}\Bigg)\Bigg] \label{eq:SBM_second_moment_eq}\\
\leq & ~ \expect{\exp\pth{\rho \sum_{i < j} \sigma_i \tilde\sigma_i \sigma_j \tilde \sigma_j }}.
\label{eq:SBM_second_moment}
\end{align}
For the Bernoulli setting where $P=\Bern(a/n)$ and $Q=\Bern(b/n)$ for 
fixed constants $a, b$,
we have $\rho \triangleq \frac{\tau}{n} + O(\frac{1}{n^2})$, where $\tau\triangleq\frac{(a-b)^2}{2(a+b)}$. 
Thus,
\[
\chi^2(\P\|\Q) + 1 \leq \expect{\exp\pth{\frac{\tau }{2n}\iprod{\sigma}{\tilde \sigma}^2  + O(1)}}.
\]
Write $\sigma=2\xi-1$, where $\xi \in \{0, 1\}^n$ is the indicator vector for the first community which is drawn uniformly at random from all binary vectors with Hamming weight $n/2$, and $\tilde \xi$ is its independent copy. Then $\Iprod{\sigma}{\tilde \sigma} = 4 \Iprod{\xi}{\tilde \xi}-n$, where 
$H \triangleq \Iprod{\xi}{\tilde \xi} \sim \Hyper(n,\frac{n}{2},\frac{n}{2})$. 
Thus
\[
\chi^2(\P\|\Q) + 1\leq \expect{\exp\pth{ \frac{\tau}{2} \pth{ \frac{4 H - n}{\sqrt{n}} }^2  + O(1)}}.
\]
Since $\frac{1}{\sqrt{n/16}}(H-n/4) \to \calN(0,1)$ as $n\diverge$ by the central limit theorem for hypergeometric distributions (see, e.g., \cite[p.~194]{FellerI}), using \cite[Theorem 1]{mgf.conv} for the convergence of moment generating function, we conclude that $\chi^2(\P\|\Q) $ is bounded if $\tau<1$. 

%
%
%

\subsubsection{Mutual information-based lower bound for correlated recovery}
It is tempting to conclude that whenever detection is impossible---that is, whenever we cannot correctly tell with high probability 
whether the observation was generated from the null or planted model---we cannot infer the planted community structure $\sigma^*$ better than chance either; this deduction, however, is not true in general 
(see \cite[Section III.D]{Banks16} for a simple counterexample). Instead, we resort
to mutual information in proving lower bounds for correlated recovery. 
In fact, there are two types of mutual information that are relevant in the context of correlated recovery:
	\paragraph{Pairwise mutual information $I(\sigma_1,\sigma_2; A)$.}
	For two communities, it is easy to show that correlated recovery is impossible if and only if 
\begin{equation}
I(\sigma_1,\sigma_2;A) = o(1)
\label{eq:MIcorr2}
\end{equation}
as $n\diverge$.
This in fact also holds for $k$ communities for any constant $k$. See \prettyref{app:MIcorr} for a justification in a general setting.
Thus \prettyref{eq:MIcorr2} provides an information-theoretic characterization for correlated recovery.

The intuition is that since $I(\sigma_1,\sigma_2;A) = I(\indc{\sigma_1 = \sigma_2};A)$, \prettyref{eq:MIcorr2} means that 
the observation $A$ does not provide enough information to distinguish whether any two vertices are in the same community. Alternatively, 
since $I(\sigma_1;A)=0$ by symmetry and $I(\sigma_1;\sigma_2)=o(1)$,\footnote{Indeed, since $\prob{\sigma_2=-|\sigma_1=+} = \frac{n}{2n-2}$, $I(\sigma_1;\sigma_2) = \log 2 - h(\frac{n}{2n-2}) = \Theta(n^{-2})$, where $h$ is the binary entropy function in \prettyref{eq:binaryentropy}.} it follows from the chain rule that 
$I(\sigma_1,\sigma_2;A)=I(\sigma_1;\sigma_2|A)+o(1)$. Thus \prettyref{eq:MIcorr2} is equivalently to stating that $\sigma_1$ and $\sigma_2$ are asymptotically independent given the observation $A$; 
this is shown in \cite[Theorem 2.1]{Mossel12} for SBM below the recovery threshold $\tau=\frac{(a-b)^2}{2(a+b)}<1$.

Based on strong data processing inequalities for mutual information, recently \cite{PW18} proposed an information-percolation method to bound the mutual information in \prettyref{eq:MIcorr2} in terms of bond percolation probabilities,
 which yields bounds or sharp recovery threshold for correlated recovery; a similar program is carried out independently in \cite{AB18} for a variant of mutual information defined via the $\chi^2$-divergence. For two communities, this method yields the sharp threshold 
in the Gaussian model but not the SBM.

Next, we describe another method of proving \prettyref{eq:MIcorr2} via second moment analysis that reaches the sharp threshold. 
Let $\P_+$ and $\P_-$ denote the conditional distribution of $A$ conditioned on $\sigma_1=\sigma_2$ and $\sigma_1\neq\sigma_2$, respectively. 
The following result can be distilled from \cite{banks-etal-colt} (see \prettyref{app:MITV} for a proof):
for any probability distribution $\Q$, if
\begin{align}
\int \frac{  \left( \P_{+}  - \P_{-} \right)^2 }{ \Q} =o(1), \label{eq:second_moment_conditional}
\end{align}
then \prettyref{eq:MIcorr2} holds and hence correlated recovery is impossible. 
The LHS of \prettyref{eq:second_moment_conditional} can be computed similarly to the usual second moment \prettyref{eq:SBM_second_moment_eq} when $\Q$ is chosen to be the distribution of $A$ under
 the null model.  
In \prettyref{app:MITV} we verify that \prettyref{eq:second_moment_conditional} is satisfied below the correlated recovery threshold $\tau=\frac{(a-b)^2}{2(a+b)}<1$ for the binary symmetric SBM.



\paragraph{Blockwise mutual information $I(\sigma; A)$.}
Although this quantity is not directly related to correlated recovery \emph{per se}, its derivative with respect to some appropriate SNR parameter can be related to or coincides with the reconstruction error thanks to the I-MMSE formula~\cite{GuoShamaiVerdu05} or variants. 
Using this method, we can prove that Kullback-Leibler (KL) divergence $D (\P \| \Q)=o(n)$ 
implies the impossibility of correlated recovery in the Gaussian case.  
As shown in \prettyref{eq:secondmoment_KL}, a bounded second moment readily implies a bounded Kullback-Leibler divergence.
Hence, as a corollary, we prove that a bounded second moment 
also implies the impossibility of correlated recovery in the Gaussian case.
Below, we sketch the proof of the impossibility of correlated recovery in the Gaussian case, by assuming $D (\P \| \Q)=o(n)$. 
The proof makes use of mutual information, the I-MMSE formula, and a type of interpolation argument~\cite{MontanariPCA14,deshpande2015asymptotic,KrzakalaXuZdeborova16}.



Assume that $A(\beta) = \sqrt{\beta} M+W$ 
in the planted model and $A=W$ in the null model, 
where $\beta \in [0,1]$ is a signal-to-noise ratio parameter, 
$
M=(\mu/\sqrt{n}) (\sigma \sigma^\top -\identity),
$ 
$W$ is a symmetric Gaussian random matrix with zero diagonal and
$W_{ij} \iiddistr \calN(0, 1)$ for all $i<j$. 
Note that $\beta=1$ corresponds to 
the binary symmetric community model in \prettyref{def:binary} with $P=\calN(\mu/\sqrt{n},1)$
and $Q=\calN( -\mu/\sqrt{n},1)$. Below we abbreviate $A(\beta)$ as $A$ whenever the
context is clear. 
First, recall that the minimum mean-squared error estimator is given by
the posterior mean of $M$:
$$
\hat{M}_{\rm MMSE}\left(A \right)  = \expect{M | A  }
$$
and the resulting (rescaled) minimum mean-squared error is
\begin{equation}
\label{eq:mmse-def}
	\MMSE(\beta) = \frac{1}{n} \,\Exp \| M - \expect{M| A } \|_{\rm F}^2 \, .
\end{equation}

We will start by proving that if $D (\P \| \Q) =o(n)$ , then
for all $\beta \in [0, 1]$, the MMSE tends to that of the trivial estimator $\hat{M} = 0$,  \ie,
\begin{align}
	\lim_{n \to \infty} \MMSE (\beta) =   \lim_{n\to \infty}  \frac{ 1}{n} \,\Exp \|M\|_{\rm F}^2  = \mu^2 . \label{eq:mmse_formula}
\end{align}
Note that $\lim_{n\to \infty} \MMSE(\beta)$ exists by~\cite[Proposition III.2]{MontanariPCA14}. 
Let us compute the mutual information between $M$ and $A$: 
\begin{align}
	I(\beta) \triangleq I(M;A)
	& = \Exp_{M, A} \, \log \frac{\P(A|M)}{\P(A)} \label{eq:Ibeta}\\
	& =  \Exp_{A} \, \log \frac{\Q(A)}{\P(A)} +  \Exp_{M, A} \, \log  \frac{ \P(A|M)}{ \Q(A) } 	\nonumber \\
	& = - D (\P \| \Q) +  \frac{1}{2} \Exp_{M, A} \left[ \sqrt{\beta } \iprod{M}{A} - \frac{ \beta \|M\|_{\rm F}^2}{2} \right] 
	\nonumber \\
	& = - D (\P \| \Q) + \frac{\beta }{4}  \Exp \|M \|_{\rm F}^2 \, .
	\label{eq:MI_formula0}
\end{align}
By assumption,  we have $D(\P \| \Q) = o(n)$ holds for $\beta=1$; 
by the data processing inequality for KL divergence~\cite{Csiszar67}, this holds for all $\beta < 1$ as well.  
Thus~\prettyref{eq:MI_formula0} becomes
\begin{align}
	\lim_{n \to \infty} \frac{1}{n} I(\beta ) = \frac{\beta }{4} \lim_{n\to \infty}  \frac{ 1}{n} \Exp\, \|M\|_{\rm F}^2 = \frac{\beta \mu^2}{4} \, .
	\label{eq:MI_formula}
\end{align}

Next we compute the MMSE.  Recall the I-MMSE formula~\cite{GuoShamaiVerdu05} for Gaussian channels:
\begin{equation}
\frac{\diff I(\beta)}{\diff \beta} = \frac{1}{2} \sum_{i<j} \left( M_{ij} - \expect{M_{ij} | A } \right)^2  
= \frac{n}{4} \MMSE(\beta) \,.
\label{eq:i-mmse}
\end{equation}
Note that the MMSE is by definition bounded above by the squared error of the trivial estimator $\hat{M} = 0$, so that for all $\beta$ we have
\begin{equation}
 \MMSE(\beta) \le \frac{1}{n} \,\Exp \, \| M \|_{\rm F}^2 \le \mu^2 \, .
 \label{eq:trivial}
\end{equation}
Combining these we have
\begin{align*}
 \frac{\mu^2}{4}   \overset{(a)}{=} \lim_{n \to \infty} \frac{I(1)}{n} 
 & \overset{(b)}{=} \frac{1}{4}  \lim_{n \to \infty} \int_{0}^{1}  \MMSE(\beta) \,\diff \beta  \\
  & \overset{(c)}{ \le} \frac{1}{4} \int_{0}^{1}  \lim_{n\to \infty}   \MMSE(\beta) \,\diff \beta  \\
  &  \overset{(d)}{\le}  \frac{1}{4} \int_{0}^{1}  \mu^2  \,\diff \beta
   = \frac{\mu^2}{4} \, , 
\end{align*}
where $(a)$ and $(b)$ hold due to~\prettyref{eq:MI_formula} and~\prettyref{eq:i-mmse}, $(c)$ follows from Fatou's lemma, and $(d)$ follows from~\prettyref{eq:trivial}, \ie, $\MMSE(\beta) \leq \mu^2$ pointwise. Since we began and ended with the same expression, these inequalities must all be equalities.  In particular, since $(d)$ holds with equality, we have \prettyref{eq:mmse_formula} holds
for almost all $\beta \in [0,1]$.  Since $\MMSE(\beta)$ is a non-increasing function of $\beta$, 
its limit $\lim_{n \to \infty}   \MMSE(\beta)$ is also non-increasing in $\beta$.  Therefore,  \prettyref{eq:mmse_formula} 
holds for all $\beta \in [0, 1]$.  This completes the proof of our claim that the optimal MMSE estimator cannot outperform the trivial one asymptotically.

To show that the optimal estimator actually converges to the trivial one, we expand the definition of $\MMSE(\beta)$ in~\prettyref{eq:mmse-def} and subtract~\prettyref{eq:mmse_formula} from it.  This gives
\begin{equation}
\lim_{n \to \infty} \frac{1}{n} \expect{ - 2\iprod{M}{\expect{M|A} } + \| \expect{M|A} \|_{\rm F}^2} = 0 \, . 
\label{eq:mmse-2}
\end{equation}
By the tower property of conditional expectation and the linearity of the inner product, it follows 
$$
	\Exp \, \iprod{M}{\expect{M|A} }
 	= \Exp \, \iprod{\expect{M|A}}{\expect{M|A}}
 	= \Exp \, \| \expect{ M|A} \|_{\rm F}^2 \, ,
$$
and combining this with~\prettyref{eq:mmse-2} gives
\begin{equation}
\label{eq:trivialest}
	\lim_{n \to \infty} \frac{1}{n} \Exp \, \| \expect{M|A} \|_{\rm F}^2 = 0 \, .
\end{equation}

Finally, for any estimator $\hat{\sigma}(A)$ of the community membership $\sigma$, we can define an estimator for $M$ by
$\hat{M} = (\mu/\sqrt{n} ) ( \hat{\sigma} \hat{\sigma}^\top - \identity)$.
Then using the Cauchy-Schwarz inequality, we have
\begin{align*}
	\Exp_{M, A} [\Iprod{M}{\hat{M}}] & = \Exp_{A} [\Iprod{ \expect{M|A} }{\hat{M}}]  \\
	& \le   \Exp_{A} \left[ \| \expect{M|A} \|_{\rm F} \| \hat{M}\|_{\rm F}  \right] \\
	& \le  \sqrt{ \Exp_{A} [ \| \expect{M|A} \|^2_{\rm F} ] } \times \mu\sqrt{n} \overset{\prettyref{eq:trivialest}}{=} o(n).
\end{align*}
Since $\Iprod{M}{\hat{M}} = \mu^2 ( \iprod{\sigma}{\hat{\sigma}}^2/n -1 )$, it follows
that $\Expect[\iprod{\sigma}{\hat \sigma}^2]=o(n^2)$ which further implies 
$ \expect{\left| \iprod{\sigma}{\hat \sigma} \right|}=o(n)$ by Jensen's inequality. 
Hence, correlated recovery of
$\sigma$ is impossible.


\medskip
In passing, we remark that while we focus on the binary symmetric community in this section, the proof techniques are widely applicable 
for many other high-dimensional inference problems such as detecting a single community~\cite{arias2013community}, 
sparse PCA, Gaussian mixture clustering~\cite{Banks16},
synchronization~\cite{PerryWeinBandeiraMoitra16}, and tensor PCA~\cite{PerryWeinBandeira16}. 
In fact, for more general $k$-symmetric community model with $P=\calN\left((k-1)\mu/\sqrt{n},1\right)$
and $Q=\calN\left(-\mu/\sqrt{n},1\right)$, the first moment method shows that both
detection and correlated recovery are information-theoretically possible when $\mu > 2 \sqrt{\log k/(k-1)}$ and
impossible when $\mu< \sqrt{2 \log (k-1)/(k-1)}$. The upper and lower bounds differ by a factor of $\sqrt{2}$ when $k$ is asymptotically
large. This gap of $\sqrt{2}$ is due to the looseness of the second moment lower bound. 
A more refined \emph{conditional} second lower
bound can be applied to show that the sharp IT threshold for detection and correlated recovery is 
$\mu=2\sqrt{\log k/k}(1+o_k(1))$ 
when $k\to \infty$~\cite{Banks16}. Complete, but not explicit, characterizations of information-theoretic reconstruction thresholds were obtained in~\cite{KrzakalaXuZdeborova16,Barbier16,LelargeMiolane16} for all finite $k$ through the Guerra interpolation technique and cavity method.

\subsection{Almost exact and exact recovery}
	\label{sec:exact}

In this subsection, we study almost exact and exact recovery using the single community model
as an illustrating example. The hidden community can be represented by its indicator vector $\xi \in \{0, 1\}^n$
such that $\xi_i=1$ if vertex $i$ is in the community and $\xi_i=0$ otherwise. 
Let $\xi^*$ denote the indicator of the true community 
and $\hat\xi=\hat\xi(A) \in \{0,1\}^n$ an estimator.
The only assumptions on the community size $K$ we impose are that $K/n$ is bounded away from one,
and, to avoid triviality, that $K\geq 2$.  Of particular interest is the case of  $K=o(n),$
where the community size grows sublinearly with respect to the network size.


\begin{definition}[Almost Exact Recovery]   \label{def:weak_recovery}
An estimator $\hat \xi$ is said to {\em almost exactly recover} $\xi^*$ if, as
$n \to \infty$,  $d_H(\xi^*, \hat\xi) / K \to 0$ in probability,  where $d_H$ denotes the Hamming distance.\index{Almost exact recovery}
\end{definition}

One can verify that the existence of an estimator satisfying \prettyref{def:weak_recovery} is equivalent to the existence of an estimator such that  $ \Expect[d_H(\xi^*, \hat\xi)] = o(K)$.

\begin{definition}[Exact Recovery]
\label{def:exact}
An estimator $\hat \xi$ {\em exactly recovers} $\xi^*,$ if, 
as $n \to \infty$,
$
\Prob[\xi^* \neq \hat\xi] \to 0,
$
where the probability is with respect to the randomness of $\xi^*$ and $A$.\index{Exact recovery}
\end{definition}

To obtain upper bounds on the thresholds for almost exact and exact recovery, we
turn to the MLE. 
Specifically,
\begin{itemize}
	\item To show the MLE achieves almost exact recovery, it suffices to prove that 
there exists $\epsilon_n=o(1)$ such that with high probability, 
$\P(A| \xi) < \P (A |\xi^*)$  for all $\xi$ with $d_H( \xi, \xi^*) \ge \epsilon_n K $. 
\item To show the MLE achieves exact recovery, it suffices to prove that
with high probability, $\P(A | \xi) < \P(A | \xi^*)$ for all $\xi \neq \xi^*$. 
\end{itemize}
This type of argument often involves two key steps. First, upper bound the probability that $\P(A| \xi) \ge \P (A |\xi^*)$ for a fixed 
$\xi$ using large deviation techniques. Second, take an appropriate union bound over all possible $\xi$ using a ``peeling'' argument which takes into account the fact that the further away $\xi$ is from $\xi^*$ the less likely for $\P(A| \xi) \ge \P (A |\xi^*)$ to occur. 
Below we discuss these two key steps in more details.

Given the data matrix $A$, a sufficient statistic for estimating the community $C^*$ is
the \emph{log likelihood ratio (LLR) matrix} $\bm{L} \in \reals^{n \times n}$, where
$L_{ij}=\log \frac{dP}{dQ}(A_{ij})$ for $i \neq j$ and $L_{ii}=0$.
For $S,T\subset [n]$,  define  
\begin{equation}
e(S,T) = \sum_{(i<j): (i,j) \in (S\times T) \cup (T\times S)} L_{ij}.
	\label{eq:eST}
\end{equation}
Let $\CML$ denote the maximum likelihood estimation (MLE) of $C^*,$ given by:
\begin{equation}
\CML=\argmax_{C\subset [n]} \{ e(C,C) : |C|= K \},
	\label{eq:MLE}
\end{equation}
which minimizes the error probability $\pprob{\widehat{C}\neq C^*}$ because $C^*$ is equiprobable by assumption.
It is worth noting that the optimal estimator that minimizes the misclassification rate (Hamming loss) is the bit-MAP decoder $\tilde\xi=(\tilde\xi_i)$, where 
$\tilde \xi_i \triangleq \argmax_{j\in\{0,1\}} \Prob[\xi_i=j|L]$. Therefore, although
the MLE is optimal for exact recovery, it need not be optimal for almost exact recovery;
nevertheless, we choose to analyze MLE due to its simplicity and it turns out to
be asymptotically optimal for almost exact recovery as well.

To state the main results, we introduce some standard notations associated
with binary hypothesis testing based on independent samples.  We assume the KL divergences $D(P\|Q)$ and $D(Q\|P)$ are finite.  
In particular, $P$ and $Q$ are mutually absolutely continuous, and the likelihood
ratio, $\frac{dP}{dQ},$  satisfies $\Expect_Q\left[ \frac{dP}{dQ} \right] = \Expect_P\left[ (\frac{dP}{dQ})^{-1} \right] =1$.
Let $L=\log \frac{dP}{dQ}$ denote the LLR.  The likelihood ratio test for
$n$ observations and threshold $n\theta$ is to declare $P$ to be the true distribution if
$\sum_{k=1}^{n} L_k \ge n\theta$ and to declare $Q$ otherwise.  
For $\theta \in [-D(Q\|P),D(P\|Q)]$,  the standard Chernoff bounds for error probability of this likelihood ratio test are given by:
\begin{align}
Q\qth{\sum_{k=1}^{n} L_k \ge n\theta} \leq \exp(-n E_Q(\theta)) \label{eq:LDupper1}
  \\
P\qth{\sum_{k=1}^{n} L_k \le n\theta} \leq \exp(-n E_P(\theta)) \label{eq:LDupper2},
\end{align}
where the log moment generating functions of $L$ are denoted by
 $\psi_Q(\lambda) = \log \Expect_Q[\exp(\lambda L)] $ and $\psi_P(\lambda) = \log \Expect_P[\exp(\lambda L)] 
= \psi_Q(\lambda + 1 )$ and the large deviations exponents are given by Legendre transforms of the
log moment generating functions:
\begin{align}
E_Q(\theta)  & = \psi_Q^*(\theta) \triangleq \sup_{\lambda \in \reals} \lambda \theta - \psi_Q(\lambda),  \label{eq:ratefunction}    \\
 E_P(\theta)  & =  \psi_P^*(\theta) \triangleq  \sup_{\lambda \in \reals} \lambda \theta - \psi_P(\lambda) =E_Q(\theta) -\theta.	 \label{eq:EP}
\end{align}
In particular, $E_P$ and $E_Q$ are convex functions. Moreover, since $\psi_Q'(0)=-D(Q\|P)$ and $\psi_Q'(1)=D(P\|Q)$, we have $E_Q(-D(Q\|P)) = E_P(D(P\|Q)) = 0$ and hence 
$E_Q(D(P\|Q))=D(P\|Q)$ and $E_P(-D(Q\|P))=D(Q\|P)$.

Under mild assumptions on the distribution $(P,Q)$ (cf.~\cite[Assumption 1]{HajekWuXu_one_info_lim15}) which are satisfied by both the Gaussian and the Bernoulli distributions, the sharp thresholds for almost exact and exact recovery under the single community model are given by the following result.

\begin{theorem}\label{thm:weak_general}
Consider the single community model with $P=\calN(\mu,1)$ and $Q=\calN(0,1)$, or 
$P=\Bern(p)$ and $Q=\Bern(q)$ with $\log \frac{p}{q}$ and $\log \frac{1-p}{1-q}$ bounded.  
If
\begin{equation}
	 K  \cdot D(P\|Q) \to \infty   \text{ and }
	 \liminf_{n\to\infty} \frac{(K-1)  D(P\|Q)}{\log \frac{n }{K}}> 2,  
	 \label{eq:weak-bdd_suff}
\end{equation}
then almost exact recovery is information-theoretically possible. If in addition to \prettyref{eq:weak-bdd_suff}, the following holds:
\begin{equation}
\liminf_{n\to\infty} \frac{K E_Q\pth{\frac{1}{K} \log \frac{n}{K}}}{\log n} > 1,
	\label{eq:voting-suff}
\end{equation}
then exact recovery is information-theoretically possible. 

Conversely, if almost exact recovery is information-theoretically possible, then 
\begin{equation}
	K \cdot D(P\|Q) \to \infty   \text{ and } \liminf_{n\to\infty} \frac{(K-1) D(P\|Q)}{\log \frac{n }{K}} \ge 2.
	\label{eq:weak-bdd_nec}
	\end{equation}
If exact recovery is information-theoretically possible, then in addition  to \prettyref{eq:weak-bdd_nec}, the following holds:
\begin{equation}
\liminf_{n\to\infty} \frac{K E_Q\pth{\frac{1}{K} \log \frac{n}{K}}}{\log n} \ge 1.
	\label{eq:voting-nec}
\end{equation}
\end{theorem}

Next we sketch the proof of \prettyref{thm:weak_general}.

\paragraph{Sufficient conditions.}
 For any $C\subset[n]$ such that $|C|=K$ and $| C \cap C^*| = \ell$, 
let $S=C^*\setminus C$ and $T=C \setminus C^*$. Then 
$$
e(C, C) - e(C^*, C^*) = e( T, T) + e(T, C^*\setminus S) - e(S, C^*). 
$$
Let $m=  \binom{K}{2}-\binom{\ell}{2}$. 
Notice that $e(S,C^*)$ has the same distribution as $ \sum_{i=1}^m L_i$ under measure $P$;
$e(T,  T) + e(T, C^\ast \backslash S)$ has the same distribution as $ \sum_{i=1}^m L_i$ under measure $Q$
where $L_i$ are i.i.d.\ copies of $\log \frac{ \diff P}{ \diff Q} $.
It readily follows from large deviation bounds \prettyref{eq:LDupper1} and \prettyref{eq:LDupper2} that 
\begin{align}
\prob{ e(T, T) + e(T, C^* \setminus S) \ge m \theta } & \ge   \exp \left(- m E_Q(\theta) \right)  \nonumber \\
\prob{ e(S, C^*)  \leq m \theta } & \le \exp \left( - m E_P(\theta) \right).  \label{eq:LDbounds}
\end{align}

Next we proceed to  describe the union bound for the proof of almost exact recovery. Note that
to show MLE achieves almost exact recovery, it is equivalent to showing
$
\prob{ | \hat{C}_{\rm ML}  \cap C^* | \le (1-\epsilon_n) K }  =o(1),
$
The first layer of union bound is straightforward:
\begin{equation}
\left\{ | \hat{C}_{\rm ML}  \cap C^* | \le (1-\epsilon_n) K  \right\} = \cup_{\ell=0}^{ \lfloor (1-\epsilon_n) K \rfloor } 
\left\{ | \hat{C}_{\rm ML}  \cap C^* | = \ell  \right\}.
\label{eq:union1}
\end{equation}
For the second layer of union bound, one na\"ive way to proceed is 
\begin{align*}
\left\{ | \hat{C}_{\rm ML}  \cap C^* | = \ell  \right\} & \subset  \{  C\in \calC_\ell: e(C, C) \geq e(C^*,C^*) \}  \\
& = \cup_{C \in \calC_\ell } \left\{  e(C, C) \geq e(C^*,C^* \right\},
\end{align*} 
where $\calC_\ell = \{ C \subset [n]: |C| = K,  |C\cap C^*| = \ell \}$. 
However, this union bound is too loose because of the high correlations 
among $e(C, C) - e(C^*, C^*)$ for different $C \in \calC_\ell$. 
Instead, we use the following union bound.
Let $\calS_\ell = \{  S \subset C^*: |S| =K-\ell \}$ and $\calT_\ell = \{ T \subset (C^*)^c: |T| =K-\ell \}.$ Then for any $\theta\in\reals$,
\begin{align}
 \{   | \hat{C}_{\rm ML}  \cap C^* | = \ell \}  & \subset  \{  \exists S \in \calS_\ell, T \in \calT_\ell:   e(S, C^*)   \le e(T, T) +  e(T, C^*\backslash S)   \} \nonumber \\
 & \subset    \{  \exists S \in \calS_\ell : e(S, C^*)  \leq m \theta  \} \nonumber \\
 & ~~~~ \cup \{ \exists  S \in \calS_\ell, T \in \calT_\ell:  e(T,  T) +  e(T, C^\ast \backslash S) \ge m \theta \} \nonumber \\
 & \subset \cup_{S \in \calS_\ell} \{ e(S, C^*)  \leq m \theta  \}  \nonumber \\
 & ~~~~ \cup_{S \in \calS_\ell, T \in \calT_\ell} \{ e(T,  T) +  e(T, C^\ast \backslash S) \ge m \theta \}. \label{eq:union2}
\end{align}
Note that we single out $e(S, C^*)$ because the number of different choices of $S$, $\binom{K}{K-\ell}$,  
is much smaller than the number of different choices of $T$, $\binom{n-K}{K-\ell}$, when $K\ll n$. 
Combining the above union bound together with the large deviation bound \prettyref{eq:LDbounds} 
yields that 
\begin{align}
\prob{ | \hat{C}_{\rm ML}  \cap C^* | =\ell} \le   \binom{K}{K-\ell} e^{- m E_P(\theta)}  + \binom{n-K}{K-\ell} \binom{K}{K-\ell} e^{- m E_Q(\theta)}. 
\label{eq:MLE_tail_analysis}
\end{align}
Note that for any $\ell \leq (1-\epsilon)K$,
\begin{align*}
 \binom{K}{K-\ell} & \le \left( \frac{K e }{K-\ell}  \right) ^{K-\ell} \le \left( \frac{e}{\epsilon} \right)^{K-\ell} \\
  \binom{n-K}{K-\ell} & \le  \left( \frac{ (n-K) e }{K-\ell}  \right) ^{K-\ell} \le  \left( \frac{(n-K) e}{K \epsilon} \right)^{K-\ell}.
\end{align*}
Hence, for  any $\ell \leq (1-\epsilon)K$,
\begin{align}
\prob{ | \hat{C}_{\rm ML}  \cap C^* | =\ell}  \le \eexp^{-(K-\ell) E_1} + \eexp^{-(K-\ell) E_2}, \label{eq:weakProbupperbound}
\end{align}
where
\begin{align*}
E_1 &\triangleq \frac{1}{2} (K-1) E_P(\theta)  - \log \frac{\eexp}{\epsilon} , \\
E_2 &\triangleq \frac{1}{2} (K-1) E_Q(\theta)   - \log \frac{(n-K) \eexp^2}{K \epsilon^2}.
\end{align*}
Thanks to the second condition in \prettyref{eq:weak-bdd_suff}, 
we have $(K-1)  D(P\|Q) (1-\eta) \geq 2  \log \frac{n}{K}$ for some $\eta \in (0,1)$. 
Choose $\theta = (1-\eta) D(P\|Q)$. Under some mild assumption on $P$ and $Q$ which is
satisfied in Gaussian and Bernoulli case, we have $E_P(\theta)  \geq c \eta^2 D(P\|Q)$ for
some universal constant $c>0$. Furthermore, recall from \prettyref{eq:EP} that $E_P(\theta)= E_Q(\theta) - \theta$. 
Hence, since $K  D(P\|Q) \to \infty $ by the assumption \prettyref{eq:weak-bdd_suff}, by choosing $\epsilon= 1/\sqrt{KD(P\|Q)},$
we have $\min\{ E_1, E_2 \} \to \infty$. The proof for almost exact recovery is complete by
taking the first layer of union bound in \prettyref{eq:union1} over $\ell$. \\

For exact recovery, we need to show $ \prob{ | \hat{C}_{\rm ML}  \cap C^* | \le K-1}  =o(1)$.
Hence, we need to further bound $\prob{ | \hat{C}_{\rm ML}  \cap C^* | =\ell} $ for any $(1-\epsilon)K \le \ell \le K-1$. 
It turns out the previous union bound \prettyref{eq:union2} is no longer tight. Instead, using
$ e(T, T) + e(T, C^*\setminus S) = e(T, T\cup C^*) - e(T, S)$, we have the following union bound
\begin{align*}
 \{   | \hat{C}_{\rm ML}  \cap C^* | = \ell \}   
 & \subset \cup_{S \in \calS_\ell} \{ e(S, C^*)  \leq m_1 \theta_1  \}  \cup_{T \in \calT_\ell} \{ e(T, T\cup C^*) \ge m_2\theta_2 \}  \\
 & ~~~~ \cup_{S \in \calS_\ell, T \in \calT_\ell} \{ e(T, S) \le m_2 \theta_2 - m_1 \theta_1 \},
\end{align*}
where $m_1 = \binom{K}{2} - \binom{\ell}{2}$, $m_2= \binom{K-\ell}{2} + (K-\ell) K$, and $\theta_1,\theta_2$ are to be optimized.
Note that we further single out $e(T, T\cup C^*)$ because it only depends on $T \in \calT_\ell$ once $C^*$ is fixed. 
Since $(1-\epsilon)K \le \ell \le K-1$, we have $|T| = |S| = K-\ell \le \epsilon K$ and thus the effect of $e(T, S)$ can be neglected.
 Therefore, approximately we can set $\theta_1=\theta_2=\theta$ and get
\begin{align}
\prob{ | \hat{C}_{\rm ML}  \cap C^* | =\ell} \lesssim  \binom{K}{K-\ell} e^{- m_1 E_P(\theta)}  + 
 \binom{n-K}{K-\ell} e^{- m_2 E_P(\theta) },
\end{align}

Using $\binom{K}{K-\ell} \le K^{K-\ell}$, $\binom{n-K}{K-\ell} \le (n-K)^{K-\ell}$, $m_2 \ge m_1 \ge (1-\epsilon) (K-\ell) K$, 
we get that for any $(1-\epsilon)K \le \ell \le K-1$,
\begin{align}
\prob{ | \hat{C}_{\rm ML}  \cap C^* |=\ell}  \le \eexp^{-(K-\ell) E_3} + \eexp^{-(K-\ell) E_4}, \label{eq:weakProbupperbound1}
\end{align}
where
\begin{align*}
E_3 &\triangleq (1-\epsilon) K E_P(\theta)  - \log K , \\
E_4 &\triangleq (1-\epsilon) K E_Q(\theta) - \log n. 
\end{align*}
Note that $E_P(\theta)= E_Q(\theta) - \theta$.  Hence, we set
$\theta= (1/K) \log (n/K)$ so that $E_3=E_4$, which goes to $+\infty$ under the assumption of \prettyref{eq:voting-suff}.
The proof of exact recovery is complete by taking the union bound over all $\ell$.

\paragraph{Necessary conditions.} 
To derive lower bounds on the  almost exact recovery threshold, we resort to
a simple rate-distortion argument. Suppose $\hat \xi$ achieves almost exact recovery
of $\xi^*$. Then $\Expect[d_H(\xi, \hat \xi)] = \epsilon_n K$ with $\epsilon_n\to 0$. 
On the one hand, consider the following chain of inequalities,
which lower bounds the amount of information required for a distortion level $\epsilon_n$:
\begin{align*}
I(A; \xi^* ) & \overset{(a)}{\geq} I(\hat \xi; \xi^* )  \geq \min_{\Expect[d(\tilde\xi,\xi^*)] \leq \epsilon_n K} I(\tilde \xi; \xi^* )   \\
 & \geq H(\xi^* )- \max_{\Expect[d(\tilde\xi,\xi^*)] \leq \epsilon_n K}  H( \tilde\xi \oplus \xi^* ) \\
& \overset{(b)}{=} 
\log \binom{n}{K} - n h\pth{\frac{\epsilon_n K}{n}} \overset{(c)} {\ge } K \log \frac{n}{K} (1+o(1)),
\end{align*}
 where $(a)$ follows from the data processing inequality for mutual information since $\xi\to A \to \hat\xi$ forms a Markov chain, $(b)$ is due 
 to the fact that $\max_{\Expect[w(X)] \leq p n } H(X) = nh(p)$ for any $p \leq 1/2$ where 
\begin{equation}
h(p) \triangleq p \log \frac{1}{p} + (1-p) \log \frac{1}{1-p}
\label{eq:binaryentropy}
\end{equation}
 is the binary entropy function and $w(x)=\sum_i x_i$,
and $(c)$ follows from the  bound $\binom{n}{K}  \geq \pth{ \frac{n}{K} }^K$, 
the assumption $K/n$ is bounded away from one, and the bound $h(p) \leq -p\log p + p$ for $p\in [0,1]$.

On the other hand, consider the following upper bound on the mutual information: 
\begin{align*}
I(A; \xi^* )
= \min_\mathbb{Q} D({ \mathbb{P}_{A|\xi^* } } \| \mathbb{Q} |\mathbb{P}_{\xi^* })	
\leq D({ \mathbb{P}_{A|\xi^* } } \|Q^{\otimes \binom{n}{2}} |\mathbb{P}_{\xi^* })  
=  	\binom{K}{2} D(P \| Q),  
\end{align*}
where the first equality follows from the geometric interpretation of mutual information as ``information radius'', see, \eg, \cite[Corollary 3.1]{PW-it};
the last equality follows from the  tensorization property of KL divergence for product distributions.
Combining the last two displays, we conclude the 
second condition in \prettyref{eq:weak-bdd_nec} is necessary for almost exact recovery. 

To show the necessity of the first condition in \prettyref{eq:weak-bdd_nec}, we can reduce almost exact recovery to a local hypothesis testing via a genie-type argument.  Given $i,j \in [n]$, let $\xi_{\backslash i,j}$ denote $\{\xi_k\colon k \neq i,j\}$.
Consider the following binary hypothesis testing problem for determining $\xi_i$.
If $\xi_i=0$, a node $J$ is randomly and uniformly chosen from $\{j\colon \xi_j=1\}$, and we observe $(A, J, \xi_{\backslash i, J })$;
if $\xi_i=1$, a node $J$ is randomly and uniformly chosen from $\{j\colon \xi_j=0 \}$, and we observe $(A, J, \xi_{\backslash i, J } )$.
It is straightforward to verify that this hypothesis testing problem is equivalent to 
testing $H_0:Q^{\otimes (K-1)} P^{\otimes (K-1)}$ versus
 $H_1: P^{\otimes (K-1)} Q^{\otimes (K-1)}$.
 Let $\calE$ denote the optimal average probability of testing error, 
 $p_{e,0}$ denote the Type-I error probability, and $p_{e,1}$ denote the Type-II error probability. 
Then we have the following chain of inequalities:
\begin{align*}
\Expect[d_H(\xi, \hat \xi)]  &  \geq \sum_{i=1}^n \min_{\hat \xi_i(A)} \Prob[\xi_i \neq  \hat \xi_i] 	  \\
& \geq  \sum_{i=1}^n \min_{\hat \xi_i ( A, J, \; \xi_{\backslash i, J})} \Prob[\xi_i \neq  \hat \xi_i]  \\
& =  n \min_{\hat \xi_1(A, J, \; \xi_{\backslash 1, J})} \Prob[\xi_1 \neq  \hat \xi_1] 	= n\calE. 
\end{align*}
By the assumption $\Expect[d_H(\xi, \hat \xi)]  =o(K)$, it follows that
$\calE =o(K/n)$. Under the assumption that $K/n$ is bounded away from one, $\calE =o(K/n)$ further 
implies that the sum of Type-I and II probabilities of error
$p_{e,0} + p_{e,1} = o(1)$, or equivalently,
$\text{TV}((P\otimes Q)^{\otimes K-1},(Q \otimes P)^{\otimes K-1}) \to 1$, where $\text{TV}(P,Q) \triangleq \int |\diff P-\diff Q|/2$ denotes the total variation distance.
Using $D(P\|Q) \geq \log \frac{1}{2(1-\text{TV}(P,Q))}$ \cite[Eqn.~(2.25)]{Tsybakov09} and the tensorization property of KL divergence for product distributions,
we conclude that 
$
(K-1)(D(P\|Q)+D(Q\|P)) \to \infty
$ 
is necessary for almost exact recovery. 
It turns out that for both the Bernoulli and Gaussian distributions as specified in
the theorem statement, $D(P\|Q) \asymp D(Q\|P)$ and hence $K D(P\|Q) \to \infty$
is necessary for almost exact recovery. 

Clearly, any estimator achieving exact recovery also achieves almost exact recovery. Hence
lower bounds for almost exact recovery hold automatically for exact recovery. 
Finally, we show the necessity of \prettyref{eq:voting-nec} for exact recovery. Since 
the MLE minimizes the error probability among all estimators if the true community $C^*$ is uniformly
distributed, it follows that if exact recovery is possible, then 
with high probability, $C^*$ has a strictly
higher likelihood than any other community $C \neq C^*$, in particular, $C=C^* \setminus \{i\} \cup \{  j\}$
for any pair of two vertices $i \in C^*$ and $j \notin C^*$. To further illustrate the proof ideas,
consider Bernoulli case of the single community model. Then $C^*$ has a strictly
higher likelihood than $C^* \setminus \{i\} \cup \{  j\}$ if and only if $e(i,C^*)$, the number of edges
connecting $i$ to vertices in $C^*$,  is larger than $e(j, C^* \setminus \{i\})$, the number of edges connecting $j$ to vertices in $C^*\setminus \{i\}$.
Therefore, with high probability, it holds that
\begin{align}
\min_{i \in C^*}  e( i , C^*) > \max_{j \notin C^*} e(j, C^* \setminus \{i_0 \} ) , \label{eq:mle_converse_cond}
\end{align}
where $i_0$ is the random index such that $i_0 \in \arg\min_{i \in C^*} e(i,C^*)$. 
Note that  $e(j, C^* \setminus \{i_0 \} )$'s are iid
for different $j \notin C^*$ and hence  a large-probability lower bound
to their maximum can be derived using inverse concentration inequalities. 
Specifically, for the sake of argument by contradiction, suppose that 
\prettyref{eq:voting-nec} does not hold. Furthermore, for ease of presentation,
assume the large deviation inequality \prettyref{eq:LDupper1} also holds in the
reverse direction (cf. \cite[Corollary 5]{HajekWuXu_one_info_lim15} for a precise
statement). Then it follows that 
$$
\prob{ e(j, C^* \setminus \{i_0 \} ) \ge \log \frac{n}{K} } \gtrsim 
\exp \left( - K E_Q \left( \frac{1}{K} \log \frac{n}{K}  \right) \right)
\ge n^{-1+\delta}
$$
for some small $\delta>0$. Since $e(j, C^* \setminus \{i_0 \} )$'s are iid and there
are $n-K$ of them, it further follows that with a large probability, 
$$
\max_{j \notin C^*} e(j, C^* \setminus \{i_0 \} ) \ge \log \frac{n}{K}.
$$
Similarly, by assuming the large deviation inequality \prettyref{eq:LDupper2} also holds in the
opposite direction and using the fact that $E_P(\theta)=E_Q(\theta)-\theta$, we get that
$$
\prob{e( i , C^*)  \le  \log \frac{n}{K} } \gtrsim \exp \left( - K E_P \left( \frac{1}{K} \log \frac{n}{K}  \right) \right)
\ge K^{-1+\delta}.
$$
Although $e(i, C^*)$'s are not independent
for different $i \in C^*$, the dependency is weak and can be controlled properly. Hence, following
the same argument as above,  we get that with a large probability, 
$$
\min_{i \in C^*}  e( i , C^*)  \le  \log \frac{n}{K}.
$$
Combining the large-probability lower and upper bounds and \prettyref{eq:mle_converse_cond} yields 
the contradiction. Hence, \prettyref{eq:voting-nec} is necessary for exact recovery.

\begin{remark}
Note that instead of using MLE, one could also apply a two-step procedure to achieve exact recovery: first use an estimator capable of almost exact recovery and then clean up the residual errors through a local voting procedure for
every vertex. Such a two-step procedure has been analyzed in \cite{HajekWuXu_one_info_lim15}.  From the 
computational perspective, for both the Bernoulli and Gaussian cases:
\begin{itemize}
\item if $K=\Theta(n)$, a linear-time degree-thresholding algorithm achieves the information limit of weak recovery (see 
\cite[Appendix A]{HajekWuXu_one_beyond_spectral15} and \cite[Appendix A]{HajekWuXu_MP_submat15});
\item if $K=\omega(n/\log n)$, whenever information-theoretically possible, exact recovery can
be achieved in polynomial time using  semi-definite programming \cite{HajekWuXu_one_sdp15};
\item if $K \geq \frac{n}{\log n} (1/(8e) + o(1))$ for Gaussian case and $K \geq \frac{n}{\log n} (\rho_{\sf BP}(p/q) + o(1))$
for Bernoulli case, exact recovery can be attained in nearly linear time via message passing plus clean up~\cite{HajekWuXu_one_beyond_spectral15,HajekWuXu_MP_submat15} whenever information-theoretically possible.
Here $\rho_{\sf BP} (p/q)$ denotes a constant only depending on $p/q$.
\end{itemize}
However, it remains open 
whether any polynomial time algorithm can achieve the respective information limit of weak recovery
for $K=o(n)$, or  exact recovery for $K \le \frac{n}{\log n} (1/(8e) -\epsilon)$ in the Gaussian case and 
for $K \le  \frac{n}{\log n} (\rho_{\sf BP}(p/q) -\epsilon)$ in the Bernoulli case, for any fixed $\epsilon>0$. \\

Similar techniques can be used to derive the almost exact and exact recovery 
thresholds for binary symmetric community model. For Bernoulli case,  
almost exact recovery  is efficiently achieved by a simple spectral method if $n(p-q)^2/(p+q) \to \infty$~\cite{yun2014adaptive},
which turns out to be also information-theoretically necessary~\cite{Mossel14}. 
Exact recovery threshold for binary community model has been derived 
and further shown be efficiently achievable by a two-step procedure consisting of 
spectral method plus clean-up~\cite{Abbe14,Mossel14}. 
For binary symmetric community model with general discrete distributions $P$ and $Q$, 
the information-theoretic limit of 
exact recovery is shown to be determined by the R\'enyi divergence of order $1/2$ between $P$ and 
$Q$~\cite{JL15}.
The analysis of MLE has been carried out under $k$-symmetric community models for general $k,$ and 
the information-theoretic exact recovery threshold has been identified in~\cite{ChenXu14} up to a universal
constant. The precise IT limit of exact recovery has been determined in~\cite{AbbeSandon15} for $k=\Theta(1)$ with
a sharp constant and further shown be to efficiently achievable by a polynomial-time two-step procedure. 
\end{remark}

\section{Computational limits}


In this section we discuss the computational limits (performance limits of all possible polynomial-time procedures) of detecting the planted structure under Planted Clique Hypothesis (to be defined later). To investigate the computational hardness of a given statistical problem, one main approach is to find an \emph{approximate randomized polynomial-time reduction}, 
which maps certain graph-theoretic problems, in particular, the \emph{planted clique} problem, to the our problem approximately in total variation, thereby showing these statistical problems are at least as hard as solving the  planted clique problem.

We focus on the single community model in \prettyref{def:single} and present results for both the submatrix detection problem 
(Gaussian) \cite{ma2013submatrix} and the community detection problem (Bernoulli) \cite{HajekWuXu14}. 
Surprisingly, under appropriate parameterizations, the two problems share the same ``easy-hard-impossible'' phase transition. 
As shown in \prettyref{fig:phase}, where the horizontal and vertical axis corresponds to the relative community size and the noise level respectively, the hardness of the detection has a sharp phase transition: optimal detection can be achieved by computationally efficient procedures for relatively large community, but provably not for small community. This is one of the first results in high-dimensional statistics where the optimal tradeoff between statistical performance and computational efficiency can be precisely quantified.
\begin{figure}[ht]%
\centering
\includegraphics[width=0.5\columnwidth]{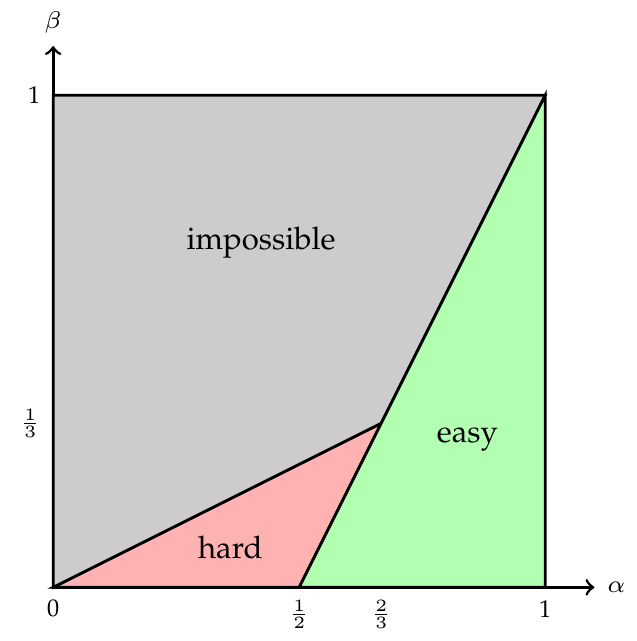}
	\caption{Computational versus statistical limits. For the submatrix detection problem, the size of the submatrix is $K = N^\alpha$ and the elevated mean is $\mu = N^{-\beta}$. For the community detection problem, the cluster size is $K=N^{\alpha}$, and the in-cluster and inter-cluster edge probability $p$ and $q $ are both on the order of $N^{-2\beta}$. }
	\label{fig:phase}	
\end{figure}
Specifically, consider the submatrix detection problem in the Gaussian case of \prettyref{def:single}, where $P=\calN(\mu,1)$ and $Q=\calN(0,1)$. In other words,  the goal is to test the null model, where the observation is an $N\times N$ Gaussian noise matrix, versus the planted model, where there exists a $K\times K$ submatrix of elevated mean $\mu$. Consider the high-dimensional setting of $K = N^\alpha$ and $\mu = N^{-\beta}$ with $N\diverge$, where $\alpha,\beta>0$ parametrizes the \emph{cluster size} and \emph{signal strength}, respectively. 
Information-theoretically, it can be shown that 
there exist detection procedures achieving vanishing error probability
if and only if $\beta < \beta^* \triangleq \max(\frac{\alpha}{2}, 2\alpha-1)$
 \cite{butucea2013}.
In contrast, if only \emph{randomized polynomial-time algorithms} are allowed, then reliable detection is impossible if $\beta > \beta^\sharp \triangleq \max(0, 2\alpha-1)$; conversely if $\beta < \beta^\sharp$, there exists a \emph{near-linear} time detection algorithm with vanishing error probability. Plotted in \prettyref{fig:phase}, the curve of $\beta^*$ and $\beta^\sharp$ corresponds to the  \textbf{statistical and computational limits} of submatrix detection respectively,  revealing the following striking phase transition: for large community ($\alpha \geq \frac{2}{3}$), optimal detection can be achieved by computationally efficient procedures; however, for small community ($\alpha < \frac{2}{3}$), computational constraint incurs a severe penalty on the statistical performance and the optimal computationally intensive procedure cannot be mimicked by any efficient algorithms. 

For the Bernoulli case, it is shown to detect a planted dense subgraph, when the in-cluster and inter-cluster edge probability $p$ and $q $ are on the same order and parameterized as $N^{-2\beta}$ and the cluster size as $K=N^{\alpha}$, the easy-hard-impossible phase transition obeys the same diagram as in \prettyref{fig:phase}~\cite{HajekWuXu14}.

Our intractability result is based on the common hardness assumption of the Planted Clique problem in the \ER graph
when the clique size is of smaller order than square root the graph cardinality \cite{Alon98}, which has been widely used to establish various hardness results in theoretical computer science  \cite{Hazan2011Nash,AAK07,KZ12,Kuvcera92,JP00,ABW10} as well as the hardness of detecting sparse principal components \cite{berthet2013lowerSparsePCA}. 
Recently, the average-case hardness of Planted Clique has been established under certain computation models \cite{Rossman10,Feldman13}
and within the sum-of-squares relaxation hierarchy~\cite{Meka15,DeshpandeMontanari15,barak-etal-planted-clique}.

The rest of the section is organized as follows:
\prettyref{sec:comp-pc} gives the precise definition of the Planted Clique problem, which forms the basis of reduction for both the submatrix detection and the community detection problem, with the latter requiring a slightly stronger assumption. 
\prettyref{sec:comp-pds} discusses how to approximately reduce the Planted Clique problem to
the single community detection problem in polynomial-time in both Bernoulli and Gaussian settings.  
Finally, \prettyref{sec:tv_bound} presents the key techniques to bound the total variation between the reduced instance and to the target hypothesis.

	\subsection{Planted Clique problem}
	\label{sec:comp-pc}

	Let $\calG(n,\gamma)$  denote the \ER graph model with $n$ vertices where each pair of vertices is connected independently with probability $\gamma$.	
	Let $\calG(n,k,\gamma)$ denote the planted clique model in which we add edges to $k$ vertices uniformly chosen from $\calG(n,\gamma)$ to form a clique.\index{Planted Clique problem}
\begin{definition}\label{def:PlantedCliqueDetection}
The \PC detection problem with parameters $(n,k,\gamma)$, denoted by $\PC(n,k,\gamma)$ henceforth, refers to the problem of testing the following hypotheses:
 \begin{align*}
H^{\rm C}_0: \quad  G \sim \mathcal{G}(n,\gamma),  \quad \quad H^{\rm C}_1: \quad  G \sim \mathcal{G}(n,k,\gamma).
\end{align*}
\end{definition}

The problem of finding the planted clique has been extensively studied for $\gamma=\frac{1}{2}$ and the state-of-the-art polynomial-time algorithms \cite{Alon98,FR00,McSherry01,Feige10findinghidden,Dekel10,ames2011plantedclique,Deshpande12} only work for $k=\Omega(\sqrt{n})$. There is no known polynomial-time solver for the \PC  problem for $k=o(\sqrt{n})$ and any constant $\gamma>0$. It is conjectured \cite{Jer92,Hazan2011Nash,JP00,AAK07,Feldman13} that the \PC problem cannot be solved in polynomial time for $k=o(\sqrt{n})$ with $\gamma =\frac{1}{2}$, which we refer to as the \PC Hypothesis.
\begin{hypothesis}[\PC Hypothesis]\label{hyp:HypothesisPlantedClique}
Fix some constant $0<\gamma \le \frac{1}{2}$. For any sequence of randomized polynomial-time tests $\{ \psi_{n,k_n} \}$ such that $\limsup_{n \to \infty} \frac{\log k_n }{ \log n} < 1/2$,
\begin{align*}
\liminf_{n \to \infty} \mathbb{P}_{H_0^{\rm C}} \{ \psi_{n,k} (G) =1 \} + \mathbb{P}_{H_1^{\rm C}} \{ \psi_{n,k}(G)=0 \} \ge 1.
\end{align*}
\end{hypothesis}
The \PC Hypothesis with $\gamma=\frac{1}{2}$ is similar to \cite[Hypothesis 1]{ma2013submatrix} and \cite[Hypothesis $\mathbf{B_{PC}}$]{berthet2013lowerSparsePCA}.
Our computational lower bounds for submatrix detection require that the \PC Hypothesis holds for $\gamma=\frac{1}{2}$ and for community detection we need to assume the \PC Hypothesis any positive constant $\gamma$.
An even stronger assumption that \PC Hypothesis holds for $\gamma=2^{-\log^{0.99} n}$ has been used in \cite[Theorem 10.3]{ABW10} for public-key cryptography.
 Furthermore, \cite[Corollary 5.8]{Feldman13} shows that under a statistical query model, any statistical algorithm requires at least $n^{\Omega(\frac{\log n}{ \log (1/\gamma) } ) }$ queries for detecting the planted bi-clique in an Erd\H{o}s-R\'enyi random bipartite graph with edge probability $\gamma$.

	\subsection{Polynomial-time randomized reduction}\index{Polynomial-time randomized reduction}
	\label{sec:comp-pds}

We present a polynomial-time randomized reduction scheme for the problem of detecting a single community (\prettyref{def:single})
in both Bernoulli and Gaussian cases. For ease of
presentation, we use the Bernoulli case as the main example, and discuss the minor 
modifications needed for the Gaussian case. The recent work~\cite{Brennan18} introduces a general reduction recipe for 
the single community detection problem under general $P, Q$ distributions, as well as various 
other detection problems with planted structures.

Let $\mathcal{G}(N,q)$ denote the Erd\H{o}s-R\'enyi random graph with $N$ vertices, where each pair of vertices is connected independently with probability $q$. Let $\mathcal{G}(N,K,p,q)$ denote the planted dense subgraph model with $N$ vertices where: (1) each vertex is included in the random set $S$ independently with probability $\frac{K}{N}$; (2) for any two vertices, they are connected independently with probability $p$ if both of them are in $S$ and with probability $q$ otherwise, where $p > q$. 
The planted dense subgraph here has a random size\footnote{We can also consider a planted dense subgraph with
a fixed size $K,$ where $K$ vertices are chosen uniformly at random to plant a dense subgraph with edge probability $p.$ 
Our reduction scheme extends to this fixed-size model; however, we have not been able to prove the distributions are
approximately matched under the alternative hypothesis. Nevertheless, the recent work~\cite{Brennan18} showed that the computational limit for detecting fixed-sized community is the same as \prettyref{fig:phase}, resolving an open problem in \cite{HajekWuXu14}.} with mean $K$, 
instead of a deterministic size $K$ as assumed in \cite{arias2013community,verzelen2013sparse}.
\begin{definition}\label{def:HypTesting}\index{Planted dense subgraph (PDS) problem}
The planted dense subgraph detection problem with parameters $(N,K,p,q)$, henceforth denoted by $\PDS(N,K,p,q)$, refers to the problem of distinguishing hypotheses:
 \begin{align*}
H_0: \quad & G \sim \mathcal{G}(N,q)\triangleq  \Prob_0,  \qquad H_1: \quad  G \sim \mathcal{G}(N,K,p,q) \triangleq  \Prob_1.
\end{align*}
\end{definition}

We aim to reduce the $\PC(n,k,\gamma)$ problem to the $\PDS(N,K,cq,q)$ problem.
For simplicity, we focus on the case of $c=2$; the general case follows similarly with a change
in some numerical constants that come up in the proof.
We are given an adjacency matrix $A \in \{0,1\}^{n \times n}$, or equivalently, a graph $G,$  and with the help of additional
randomness, will map it to an adjacency matrix $\tA \in \{0,1\}^{N \times N},$ or equivalently, a graph $\tG$
such that the hypothesis $H_0^{\rm C}$ (resp.\ $H_1^{\rm C}$) in \prettyref{def:PlantedCliqueDetection} is mapped to $H_0$ exactly (resp.\ $H_1$ approximately) in \prettyref{def:HypTesting}. In other words, if $A$ is drawn from $\calG(n,\gamma)$, then $\tA$ is distributed according to $\Prob_0$; If $A$ is drawn from $\calG(n,k,1,\gamma)$, then the distribution of $\tA$ is close in total variation to $\Prob_1$.

Our reduction scheme works as follows. Each vertex in $\tG$ is randomly assigned a parent vertex in $G,$
with the choice of parent being made independently for different vertices in $\tG,$  and uniformly
over the set $[n]$ of vertices in $G.$  Let $V_s$ denote the set of vertices in $\tG$ with parent
$s\in [n]$  and let $\ell_s=|V_s|$. Then the set of children nodes $\{V_s: s \in [n]\}$ form a random partition of $[N]$.
For any $1 \leq s \leq t \leq n,$ the number of edges, $E(V_s,V_t)$, from vertices in $V_s$ to vertices in $V_t$
in $\tG$ will be selected randomly with a conditional probability distribution specified below.
Given  $E(V_s,V_t),$  the particular  set of edges with cardinality $E(V_s,V_t)$ is chosen uniformly at
random.

It remains to specify, for $1\leq s \leq t \leq n,$
the conditional distribution of $E(V_s,V_t)$ given $\ell_s, \ell_t,$ and $A_{s,t}.$
Ideally, conditioned on $\ell_s$ and $\ell_t$, we want to construct
a Markov kernel from $A_{s,t}$ to $E(V_s,V_t)$ which maps  $\Bern(1)$  to the desired edge distribution $\Binom(\ell_s\ell_t,p)$, and $\Bern(1/2)$  to $\Binom(\ell_s\ell_t,q)$, depending on whether both $s$ and $t$ are in the clique or not, respectively. Such a kernel, unfortunately, provably does not exist. Nevertheless, this objective can be accomplished approximately in terms of the total variation. For $s=t \in [n],$  let $E(V_s,V_t)\sim\Binom( \binom{\ell_t}{2}, q).$
For $ 1 \le s < t \le n$, denote $P_{\ell_s \ell_t} \triangleq \Binom(\ell_s \ell_t,p)$ and $Q_{\ell_s \ell_t} \triangleq \Binom(\ell_s \ell_t,q)$.
Fix $0 < \gamma \leq \frac{1}{2}$ and put $m_0 \triangleq \lfloor \log_2 (1/\gamma) \rfloor$.
Define
    \begin{align*}
    P'_{\ell_s \ell_t} (m)=  \left\{
    \begin{array}{rl}
    P_{\ell_s \ell_t} (m) + a_{\ell_s \ell_t} & \text{for } m=0,\\
    P_{\ell_s \ell_t}(m) & \text{for } 1 \le m \le m_0, \\
    \frac{1}{\gamma} Q_{\ell_s \ell_t} (m) & \text{for } m_0 < m \le \ell_s \ell_t .
    \end{array} \right.
    \end{align*}
    where $a_{\ell_s \ell_t}=\sum_{m_0<m \le  \ell_s\ell_t} [ P_{\ell_s \ell_t}(m) - \frac{1}{\gamma} Q_{\ell_s \ell_t}(m) ]$.
    Let $Q'_{\ell_s \ell_t} = \frac{1}{1-\gamma} (Q_{\ell_s \ell_t} - \gamma P'_{\ell_s \ell_t})$.
The idea behind our choice of $P'_{\ell_s \ell_t}$ and $Q'_{\ell_s \ell_t}$ is as follows. 
For a given $P'_{\ell_s \ell_t}$, we choose $Q'_{\ell_s \ell_t}$ to
map $\Bern(\gamma)$ to $\Binom(\ell_s\ell_t,q)$ exactly;
however, for $Q'$ to be a well-defined
probability distribution, we need to ensure that $Q_{\ell_s \ell_t} (m) \ge \gamma P'_{\ell_s \ell_t}(m)$, which fails when $m \leq m_0$. Thus,
we set $P'_{\ell_s \ell_t} (m) = Q_{\ell_s \ell_t} (m) /\gamma$ for $m>m_0$. The remaining probability mass $a_{\ell_s \ell_t}$ is added to $P_{\ell_s \ell_t} (0)$ so that $P'_{\ell_s \ell_t}$ is a well-defined probability distribution.

It is straightforward to verify that $Q'_{\ell_s \ell_t}$ and $P'_{\ell_s \ell_t}$ are well-defined probability distributions, and
\begin{align}
d_{\rm TV} ( P'_{\ell_s \ell_t}, P_{\ell_s\ell_t} ) \le  4 (8q \ell^2)^{(m_0+1)}. \label{eq:TV_bound}
\end{align}
as long as $\ell_s, \ell_t \le 2 \ell$ and $16 q \ell^2 \le 1$, where $\ell=N/n$.  
Then, for $1\leq s < t \leq n,$ the conditional distribution of $E(V_s,V_t)$ given $\ell_s,\ell_t,$ and $A_{s,t}$ is given by
    \begin{align}
    E(V_s, V_t) \sim \left\{
    \begin{array}{rl}
    P'_{\ell_s \ell_t} & \text{if } A_{st}=1, \ell_s, \ell_t \le 2\ell\\
    Q'_{\ell_s \ell_t} & \text{if } A_{st}=0, \ell_s, \ell_t \le 2 \ell \\
    Q_{\ell_s \ell_t} & \text{if } \max\{ \ell_s, \ell_t \}> 2 \ell.
    \end{array} \right. \label{eq:edgedist}
    \end{align}

Next we show that the randomized reduction defined above maps $\calG(n,\gamma)$ into $\calG(N,q)$ under the null hypothesis and $\calG(n,k,\gamma)$ approximately  into $\calG(N,K,p,q)$ under the alternative hypothesis, respectively.
By construction, $(1-\gamma) Q'_{\ell_s \ell_t}+ \gamma P'_{\ell_s \ell_t} = Q_{\ell_s \ell_t}= \Binom(\ell_s \ell_t,q)$ and therefore the null distribution of the \PC problem is exactly matched to that of the \PDS problem, i.e., $P_{\tG|H_0^C}=\Prob_0$.
The core of the proof lies in establishing that the alternative distributions are approximately matched.
The key observation is that by \prettyref{eq:TV_bound}, $P'_{\ell_s \ell_t}$ is close to $P_{\ell_s \ell_t}= \Binom(\ell_s \ell_t,p)$ and thus for nodes with distinct parents $s \neq t$ in the planted clique, the number of edges $E(V_s,V_t)$ is approximately distributed as the desired $\Binom(\ell_s \ell_t,p)$; for nodes with the same parent $s$ in the planted clique, even though $E(V_s,V_s)$ is distributed as $\Binom(\binom{\ell_s}{2},q)$ which is not sufficiently close to the desired $\Binom(\binom{\ell_s}{2},p)$, after averaging over the random partition $\{V_s\}$, the total variation distance becomes negligible.  More formally, we have the following proposition; the proof is postponed to
the next subsection.

\begin{proposition}\label{prop:reduction}
Let $\ell, n \in \naturals$, $k \in [n]$ and $\gamma \in (0,\frac{1}{2} ]$. Let $N= \ell n$, $K=k\ell$, $p=2q$ and $m_0= \lfloor \log_2 (1/\gamma) \rfloor$. Assume that $16 q \ell^2 \le  1$ and $k \geq 6e \ell$.
If $G \sim \mathcal{G}(n, \gamma)$, then $\tG \sim \calG(N,q)$, \ie, $P_{\tG|H_0^C}=\Prob_0$.
If $G \sim \calG(n,k,1,\gamma)$, then
\begin{align}
& d_{\rm TV} \left(P_{\tG|H_1^C}, \Prob_1 \right)  \nonumber \\
& \le e^{-\frac{K}{12}} + 1.5 k e^{-\frac{\ell}{18}} + 2 k^2 (8q\ell^2)^{m_0+1} + 0.5 \sqrt{e^{72e^2 q \ell^2} -1} + \sqrt{0.5k} e^{-\frac{\ell}{36}}. \label{eq:defxi}
\end{align}
\end{proposition}

\paragraph{Reduction scheme in the Gaussian case}
The same reduction scheme can be tweaked slightly to work for the Gaussian case, which, in fact, only needs the \PC hypothesis for $\gamma = \frac{1}{2}$.\footnote{The original proof in \cite{ma2013submatrix} for the submatrix detection problem crucially relies on the Gaussianity of the reduction maps a bigger planted clique instance into a smaller instance for submatrix detection by means of averaging.}
In this case, we aim to map an adjacency matrix $A \in \{0,1\}^{n \times n}$ 
to a symmetric data matrix $\tA \in \reals^{N \times N}$ with zero diagonal, or equivalently, a 
\emph{weighted complete} graph $\tG$.

For any $1 \leq s \leq t \leq n,$ we let $E(V_s, V_t)$
denote the average weights of edges between $V_s$ and $V_t$ in $\tG$. 
Similar to the Bernoulli model, we will first generate 
$E(V_s, V_t)$ randomly with a properly chosen conditional probability distribution. 
Since $E(V_s,V_t)$ is a sufficient statistic for the set of Gaussian edge weights, the specific weight assignment 
can be generated from the average weight using the same kernel for both the null and the alternative. 

To see how this works, consider a general setup where 
$X_1,\ldots,X_n \iiddistr \calN(\mu,1)$. Let $\bar X = (1/n) \sum_{i=1}^n X_i$. 
Then we can simulate $X_1,\ldots,X_n$ based on the sufficient statistic $\bar X$ as follows. 
Let $[v_0,v_1,\ldots,v_{n-1}]$ be an orthonormal basis for $\reals^n$, with $v_0= \frac{1}{\sqrt{n}} \ones$ and $\ones = (1,\ldots,1)^\top$.
Generate $Z_1,\ldots,Z_{n-1}\iiddistr \calN(0,1)$. Then $\bar X \ones + \sum_{i=1}^{n-1} Z_i v_i \sim \calN(\mu \ones,I_n)$.
Using this general procedure, we can generate the weights 
$\tA_{V_s,V_t}$ 
based on $E(V_s, V_t)$.

It remains to specify, for $1\leq s \leq t \leq n,$
the conditional distribution of $E(V_s,V_t)$ given $\ell_s, \ell_t,$ and $A_{s,t}.$
Similar to the Bernoulli case, conditioned on $\ell_s$ and $\ell_t$, ideally we would want to find
a Markov kernel from $A_{s,t}$ to $E(V_s,V_t)$ which maps  
$\Bern(1)$  to the desired distribution $\calN(\mu, 1/\ell_s \ell_t)$, 
and $\Bern(1/2)$  to $\calN(0,1/\ell_s\ell_t)$, 
depending on whether both $s$ and $t$ are in the clique or not, respectively. 
This objective can be accomplished approximately in terms of the total variation. 
For $s=t \in [n],$  let $E(V_s,V_t)\sim \calN( 0, 1/\ell_s\ell_t).$
For $ 1 \le s < t \le n$, denote 
$P_{\ell_s \ell_t} \triangleq \calN( \mu, 1/\ell_s \ell_t)$ and 
$Q_{\ell_s \ell_t} \triangleq \calN(0, 1/\ell_s \ell_t)$, 
with density function $p_{\ell_s\ell_t}(x)$ and $q_{\ell_s\ell_t}(x)$, respectively.

Fix $\gamma = \frac{1}{2}$.
Note that
\[
\frac{q_{\ell_s\ell_t}(x)}{p_{\ell_s\ell_t}(x)} = \exp \left[ \ell_s \ell_t \mu ( \mu/2 -x) \right] \geq \gamma
\]
if and only if $x \leq x_0 \triangleq \frac{\mu}{2}+\frac{1}{\mu\ell_s\ell_t} \log \frac{1}{\gamma}$.
Therefore, we define $P'_{\ell_s\ell_t}$ and $Q'_{\ell_s\ell_t}$ with the following density:
$q'_{\ell_s \ell_t} = \frac{1}{1-\gamma} (q_{\ell_s \ell_t} - \gamma p'_{\ell_s \ell_t})$ and 
  \begin{align*}
    p'_{\ell_s \ell_t} (x)= 
		\left\{
    \begin{array}{rl}
    p_{\ell_s \ell_t}(x) + f_{\ell s \ell_t} (2\mu-x ) & \text{for } x < 2\mu- x_0,\\
    p_{\ell_s \ell_t}(x) & \text{for } x \le x_0, \\
    \frac{1}{\gamma} q_{\ell_s \ell_t} (x) & \text{for } x>x_0.
    \end{array} \right.
    \end{align*}
    where $f_{\ell s \ell_t} (x) = p_{\ell_s \ell_t}(x) - \frac{1}{\gamma} q_{\ell_s \ell_t}(x)$.
    Let
    $$
    a_{\ell_s \ell_t}=\int_{x_0}^{\infty} f_{\ell s \ell_t} (x) dx
    \leq \bar\Phi \left( - \frac{\mu}{2}\sqrt{\ell_s\ell_t} + \frac{1}{\mu\sqrt{\ell_s\ell_t}} \log \frac{1}{\gamma} \right).
    $$	
Similar to the Bernoulli case, it is straightforward to verify that $Q'_{\ell_s \ell_t}$ and $P'_{\ell_s \ell_t}$ are well-defined probability distributions, and
\begin{align}
d_{\rm TV} ( P'_{\ell_s \ell_t}, P_{\ell_s\ell_t} )
=a_{\ell_s \ell_t} 
 \le \bar\Phi \left(\frac{1}{2\mu\sqrt{\ell_s\ell_t}} \log \frac{1}{\gamma} \right) 
\leq \exp\left(- \frac{1}{32 \mu^2 \ell^2} \log^2 \frac{1}{\gamma} \right). \label{eq:TV_bound-gaussian}
\end{align}
as long as $\ell_s, \ell_t \le 2 \ell$ and $4 \mu^2 \ell^2 \le \log (1/\gamma)$, where $\ell=N/n$. 
Following the same argument as Bernoulli case, we can obtain a counterpart to \prettyref{prop:reduction}.
\begin{proposition}\label{prop:reduction_gaussian}
Let $\ell, n \in \naturals$, $k \in [n]$ and $\gamma=1/2$. Let $N= \ell n$ and 
$K=k\ell$ Assume that 
$16 \mu^2 \ell \le 1$ and $k \geq 6e \ell$.
Let $\Prob_0$ and $\Prob_1$ denote the desired null and alternative distributions of the submatrix detection problem $(N,K,\mu)$. 
If $G \sim \mathcal{G}(n, \gamma)$, then  $P_{\tG|H_0^C}=\Prob_0$.
If $G \sim \calG(n,k,1,\gamma)$, then
\begin{align}
& d_{\rm TV} \left(P_{\tG|H_1^C}, \Prob_1 \right)  \nonumber \\
& \le e^{-\frac{K}{12}} + 1.5 k e^{-\frac{\ell}{18}} + 
\frac{k^2}{2} \exp \left( - \frac{ \log^2 2 }{32 \mu^2 \ell^2} \right)
+ 0.5 \sqrt{e^{72e^2 \mu^2 \ell^2} -1} + \sqrt{0.5k} e^{-\frac{\ell}{36}}. \label{eq:defxi_gaussian}
\end{align}
\end{proposition}


Let us close this section with two remarks. First, to investigate the computational aspect of inference in the Gaussian model, since the computational complexity is not well-defined for tests dealing with samples drawn from non-discrete distributions, which cannot be represented by finitely many bits almost surely. 
To overcome this hurdle, we consider a sequence of \emph{discretized} Gaussian models that is \emph{asymptotically equivalent} to the original model in the sense of Le Cam~\cite{Lecam86} and hence preserves the statistical difficulty of the problem. 
In other words, the continuous model and its appropriately discretized counterpart are statistically indistinguishable and, more importantly, the computational complexity of tests on the latter are well-defined. More precisely, for the submatrix detection model, provided that each entry of the $n\times n$ matrix $A$ is quantized by $\Theta(\log n)$ bits, the discretized model is asymptotically equivalent to the previous model (cf.~\cite[Section 3 and Theorem 1]{ma2013submatrix} for a precise bound on the Le Cam distance).
With a slight modification, the above reduction scheme can be made to the discretized model (cf.~\cite[Section 4.2]{ma2013submatrix}).

Second, we comment on the distinctions between the reduction scheme here and the prior work that relies on planted clique as the hardness assumption. Most previous work \cite{Hazan2011Nash,AAK07,Alon11,ABW10} in the theoretical computer science
literature uses the reduction from the \PC  problem to generate computationally hard instances of other problems and establish \emph{worst-case} hardness results; the underlying distributions of the instances could be arbitrary. The idea of proving hardness of a hypothesis testing problem by means of approximate reduction from the planted clique problem such that the reduced instance is close to the target hypothesis in total variation originates from the seminal work by \cite{berthet2013lowerSparsePCA} and the subsequent paper by \cite{ma2013submatrix}.
The main distinction between these work and the results presented in this article based on the techniques in \cite{HajekWuXu14} is that  \cite{berthet2013lowerSparsePCA} studied a composite-versus-composite testing problem and \cite{ma2013submatrix} studied a simple-versus-composite testing problem, both in the minimax sense, as opposed to the simple-versus-simple hypothesis considered here and in \cite{HajekWuXu14}, which constitutes a stronger hardness result. For composite hypothesis, a reduction scheme works as long as the distribution of the reduced instance is close to \emph{some} mixture distribution under the hypothesis. This freedom is absent in constructing reduction for simple hypothesis, which renders the reduction scheme as well as the corresponding calculation of total variation considerably more difficult.  In contrast, for simple-versus-simple hypothesis, the underlying distributions of the problem instances generated from the reduction must be close to the desired distributions in total variation under both the null and alternative hypotheses.

\subsection{Bounding the total variation distance}
\label{sec:tv_bound}

Below we prove \prettyref{prop:reduction} and obtain the desired computational limits given by \prettyref{fig:phase}. We only consider the Bernoulli case as the derivations for Gaussian case are analogous. The main technical challenge is bounding the total variation distance in \prettyref{eq:defxi}. 

\begin{proof}[Proof of \prettyref{prop:reduction}]
Let $[i,j]$ denote the unordered pair of $i$ and $j$.
For any set $I \subset [N]$, let $\calE(I)$ denote the set of unordered pairs of distinct elements in $I$, i.e., $\calE(I)=\{[i,j]: i,j \in S, i\neq j\}$, and let $\calE(I)^c=\calE( [N] ) \setminus \calE(I)$.
For $s, t \in [n]$ with $s \neq t$, let $\tG_{V_sV_t}$ denote the bipartite graph where the set of left (right) vertices is $V_s$ (resp.\xspace $V_t$) and the set of edges is the set of edges in $\tG$ from vertices in $V_s$ to vertices in $V_t$. For $s \in [n]$, let $\tG_{V_sV_s}$ denote the subgraph of $\tG$ induced by $V_s$. Let $\tP_{V_sV_t}$ denote the edge distribution of $\tG_{V_sV_t}$ for $s, t \in [n]$.

It is straightforward to verify that the null distributions are exactly matched by the reduction scheme. Henceforth, we 
consider the alternative hypothesis, under which $G$ is drawn from the planted clique model $\mathcal{G}(n,k,\gamma)$. Let $C \subset [n]$ denote the planted clique. Define $S=\cup_{t \in C} V_t$ and recall $K=k \ell$. Then $|S| \sim \Binom(N, K/N)$ and conditional on $|S|$, $S$ is uniformly distributed over all possible subsets of size $|S|$ in $[N]$.
By the symmetry of the vertices of $G$, the distribution of $\tA$ conditional on $C$ does not depend on $C$. Hence, without loss of generality, we shall assume that $C=[k]$ henceforth.
The distribution of $ \widetilde{A}$ can be written as a mixture distribution indexed by the random set $S$ as
\begin{align*}
\widetilde{A} \sim \widetilde{\Prob}_1 \triangleq \mathbb{E}_{S} \left[\tP_{SS} \times   \prod_{ [i,j] \in \calE(S)^c } \Bern(q) \right].
\end{align*}
By the definition of $\Prob_1$,
\begin{align}
& d_{\rm TV} ( \widetilde{\Prob}_1, \Prob_1 ) \nonumber \\
&= d_{\rm TV} \left( \mathbb{E}_{S} \left[ \tP_{SS} \times \prod_{[i,j] \in \calE(S)^c } \Bern(q) \right],    \mathbb{E}_{S} \left[  \prod_{[i,j] \in \calE(S) } \Bern(p) \prod_{[i,j] \in \calE(S)^c } \Bern(q)   \right] \right) \nonumber \\
& \le \mathbb{E}_{S} \left[  d_{\rm TV} \left(  \tP_{SS} \times \prod_{[i,j] \in \calE(S)^c } \Bern(q), \prod_{[i,j] \in \calE(S) } \Bern(p) \prod_{[i,j] \in \calE(S)^c } \Bern(q)   \right) \right] \nonumber\\
& =\mathbb{E}_{S} \left[  d_{\rm TV} \left( \tP_{SS}, \prod_{[i,j] \in \calE(S) } \Bern(p)  \right) \right] \nonumber\\
& \le \mathbb{E}_{S} \left[  d_{\rm TV} \left( \tP_{SS}, \prod_{[i,j] \in \calE(S)} \Bern(p)  \right) \1{ |S| \le 1.5K }\right] + \exp (-K/12), \label{eq:TotalvariationDistance}
\end{align}
where the first inequality follows from the convexity of $(P,Q) \mapsto d_{\rm TV} (P,Q)$, and the last inequality follows from applying the Chernoff bound to $|S|$.
Fix an $S \subset [N]$ such that $|S| \le 1.5 K$. Define $P_{V_tV_t}=\prod_{[i,j] \in \calE(V_t)} \Bern(q)$ for $t \in [k]$ and
$P_{V_sV_t}= \prod_{(i,j) \in V_s \times V_t} \Bern(p)$ for $1 \le s<t \le k$.
By the triangle inequality,
\begin{align}
& d_{\rm TV} \left( \tP_{SS},
\prod_{[i,j] \in \calE(S)} \Bern(p)   \right) \nonumber \\
\le & ~d_{\rm TV} \left( \tP_{SS},  \mathbb{E}_{V_1^k } \left [  \prod_{1 \le s \le t\le k} P_{V_sV_t} \biggm \vert S \right]\right) \label{eq:TotalVarTriangle1} \\
& ~+d_{\rm TV}  \left( \mathbb{E}_{V_1^k} \left [  \prod_{1 \le s \le t\le k} P_{V_sV_t} \biggm \vert S \right], \prod_{[i,j] \in \calE(S)} \Bern(p) \right). \label{eq:TotalVarTriangle2}
\end{align}

To bound the term in \prettyref{eq:TotalVarTriangle1}, first note that conditioned on the set $S$, $\{V_1^k\}$ can be generated as follows: Throw balls indexed by $S$ into bins indexed by $[k]$ independently and uniformly at random; let $V_t$ is the set of balls in the $t^\Th$ bin. Define the event $E = \{V_1^k: |V_t| \leq 2 \ell, t \in[k]\}$. Since $|V_t| \sim \Binom(|S|,1/k)$ is stochastically dominated by $\Binom(1.5K,1/k)$ for each fixed $ 1 \le t\le k$, it follows from the Chernoff bound and the union bound that $\Prob\{E^c\} \le k \exp(-\ell/18)$.
\begin{align*}
& d_{\rm TV} \left(  \tP_{SS},  \mathbb{E}_{V_1^k} \left [  \prod_{1 \le s \le t \le k} P_{V_sV_t} \biggm \vert S \right]\right) \nonumber\\
&\overset{(a)}{=} d_{\rm TV} \left(  \mathbb{E}_{V_1^k} \left [  \prod_{1 \le s \le t \le k} \widetilde{P}_{V_s V_t}  \biggm \vert S\right],  \mathbb{E}_{V_1^k} \left [  \prod_{1 \le s \le t \le k} P_{V_sV_t} \biggm \vert S \right]\right) \nonumber \\
& \le \mathbb{E}_{V_1^k } \left [ d_{\rm TV} \left(  \prod_{1 \le s \le t \le k} \widetilde{P}_{V_sV_t},   \prod_{1 \le s \le t \le k} P_{V_sV_t} \right) \biggm \vert S \right] \nonumber\\
& \le \mathbb{E}_{V_1^k } \left [ d_{\rm TV} \left( \prod_{1 \le s \le t \le k} \widetilde{P}_{V_sV_t},   \prod_{1 \le s \le t \le k} P_{V_sV_t} \right)  \1{V_1^k \in E}\biggm \vert S \right] +k \exp(-\ell/18),
\end{align*}
where $(a)$ holds because conditional on $V_1^k$, $\{ \tA_{V_s V_t} : s, t \in [k] \}$ are independent.  Recall that $\ell_t=|V_t|$. For any fixed $V_1^k \in E$, we have
\begin{align*}
 & d_{\rm TV} \left( \prod_{1 \le s \le t \le k} \widetilde{P}_{V_sV_t},   \prod_{1 \le s \le t \le k} P_{V_sV_t} \right) \\
 &\overset{(a)}{=} d_{\rm TV} \left( \prod_{1\leq s<t \leq k} \widetilde{P}_{V_sV_t},   \prod_{1\leq s<t \leq k} P_{V_sV_t} \right)  \nonumber  \\
 &\overset{(b)}{=} d_{\rm TV} \left( \prod_{1\leq s<t \leq k} P'_{\ell_s \ell_t},  \prod_{1\leq s<t \leq k} P_{\ell_s \ell_t} \right) \nonumber  \\
& \le  d_{\rm TV} \left( \prod_{1\leq s<t \leq k} P'_{\ell_s \ell_t},  \prod_{1\leq s<t \leq k} P_{\ell_s \ell_t} \right) \nonumber \\
& \le \sum_{ 1 \le s< t \le k} d_{\rm TV} \left(   P'_{\ell_s \ell_t},  P_{\ell_s \ell_t} \right) \overset{(c)}{\le} 2 k^2 (8q \ell^2)^{(m_0+1)},
\end{align*}
where $(a)$ follows since $\tP_{V_tV_t}=P_{V_tV_t}$ for all $t \in [k]$;
 $(b)$ is because the number of edges $E(V_s,V_t)$ is a sufficient statistic for testing $\tP_{V_sV_t}$ versus $P_{V_sV_t}$ on the submatrix $A_{V_sV_t}$ of the adjacency matrix; $(c)$ follows from the total variation
 bound \prettyref{eq:TV_bound}. 
 Therefore,
\begin{align}
d_{\rm TV} \left(  \tP_{SS},  \mathbb{E}_{V_1^k} \left [  \prod_{1 \le s \le t \le k} P_{V_sV_t} \biggm \vert S \right]\right) \le 2 k^2 (8q \ell^2)^{(m_0+1)} +k \exp(-\ell/18).\label{eq:TotalVariationNondiagonal}
\end{align}

To bound the term in \prettyref{eq:TotalVarTriangle2}, applying \cite{HajekWuXu14}[Lemma 9], which is a conditional version of the second moment method, yields
\begin{align}
& d_{\rm TV}  \left( \mathbb{E}_{V_1^k } \left [ \prod_{1 \le s \le t \le k} P_{V_sV_t}\biggm \vert S \right],\prod_{[i,j] \in \calE(S) } \Bern(p)  \right) \nonumber\\
& \le \frac{1}{2} \prob{E^c} + \frac{1}{2}
\sqrt{\mathbb{E}_{  V_1^k; \tV_1^k} \left[ g (V_1^k , \tV_1^k) \indc{V_1^k \in E}\indc{\tV_1^k \in E} \biggm \vert S  \right] -1 + 2 \prob{E^c}}, \label{eq:mixturebound}
\end{align}
where
\begin{align*}
g (V_1^k , \tV_1^k) &= \int \frac{\prod_{1 \le s \le t \le k} P_{V_sV_t} \prod_{1 \le s \le t \le k} P_{\tV_s \tV_t}} {\prod_{[i,j]\in \calE(S)} \Bern(p) } \\
& =  \prod_{s,t=1}^k \left( \frac{q^2}{p} + \frac{(1-q)^2}{1-p} \right)^{ \binom{|V_s \cap \tV_t | }{2} } = \prod_{s,t=1}^k\left( \frac{1-\frac{3}{2} q }{1-2q} \right)^{ \binom{|V_s \cap \tV_t | }{2} }.
\end{align*}
Let $X \sim \text{Bin}(1.5K, \frac{1}{k^2})$ and $Y \sim \text{Bin}( 3 \ell, e/k)$. It follows that
\begin{align}
& \mathbb{E}_{  V_1^k; \tV_1^k} \left[ \prod_{s,t=1}^k \left( \frac{1-\frac{3}{2} q }{1-2q} \right)^{ \binom{|V_s \cap \tV_t | }{2} } \prod_{s,t=1}^k \indc{|V_s| \leq 2\ell, |\tV_t| \leq 2\ell} \biggm \vert S \right] \nonumber \\
& \overset{(a)}{\leq}   \mathbb{E}_{  V_1^k; \tV_1^k} \left[ \prod_{s,t=1}^k e^{ q \binom{|V_s \cap \tV_t | \wedge 2\ell}{2} } \biggm \vert S \right] \nonumber \\
& \stackrel{(b) }{\leq}  ~ \prod_{s,t=1}^k \expect{ e^{ q \binom{|V_s \cap \tV_t | \wedge 2\ell }{2} }  \biggm \vert S }	\nonumber \\
 & \overset{(c)}{ \le} \left( \mathbb{E} \left[ e^{q\binom{X \wedge 2\ell}{2}} \right ]  \right)^{k^2} \overset{(d)} {\le}    \mathbb{E} \left[ e^{q\binom{Y}{2}} \right]^{k^2} \overset{(e)}{\le}  \exp(72 e^2 q\ell^2), \label{eq:boundchisquare}
\end{align}
where $(a)$ follows from $1+x \le e^x$ for all $x\ge 0$ and $q<1/4$;
$(b)$ follows from the negative association property of $\{|V_s \cap \tV_t |: s,t\in[k]\}$ proved in
\cite{HajekWuXu14}[Lemma 10] in view of the monotonicity of $x \mapsto e^{q\binom{x \wedge 2\ell}{2}}$ on $\reals_+$; $(c)$ follows because $|V_s \cap \tV_t |$ is stochastically dominated by $\Binom(1.5K, 1/k^2)$ for all $(s, t) \in [k]^2$; $(d)$ follows from 
\cite{HajekWuXu14}[Lemma 11]);
$(e)$ follows from \cite{HajekWuXu14}[Lemma 12] with $\lambda=q/2$ and $q \ell \le 1/8$. Therefore, by \prettyref{eq:mixturebound}
\begin{align}
d_{\rm TV}  \left(  \tP_{SS},\prod_{[i,j] \in \calE(S) } \Bern(p)  \right) & \le0.5 k e^{-\frac{\ell}{18}} + 0.5 \sqrt{ e^{72e^2 q \ell^2} -1 + 2  k e^{-\frac{\ell}{18}}} \nonumber\\
& \le 0.5 k e^{-\frac{\ell}{18}} + 0.5 \sqrt{e^{72e^2 q \ell^2} -1} + \sqrt{0.5k} e^{-\frac{\ell}{36}}. \label{eq:TotalVariationDiagonal}
\end{align}
\prettyref{prop:reduction} follows by combining \prettyref{eq:TotalvariationDistance}, \prettyref{eq:TotalVarTriangle1}, \prettyref{eq:TotalVarTriangle2}, \prettyref{eq:TotalVariationNondiagonal} and \prettyref{eq:TotalVariationDiagonal}.
\end{proof}

The following theorem establishes the computational hardness of the $\PDS$ problem in the interior of the red region in \prettyref{fig:phase}. 
\begin{theorem}
\label{thm:pds}
   Assume \PC Hypothesis (\prettyref{hyp:HypothesisPlantedClique}) holds for all $0<\gamma \le 1/2$.  
Let $ \alpha>0$ and $0<\beta<1$ be such that
\begin{align}
  \max \{ 0, 2\alpha-1 \} \triangleq \beta^{\sharp} <\beta < \frac{\alpha}{2}.
  \label{eq:computationallimit}
    \end{align}
    Then there exists a sequence $\{(N_\ell,K_\ell,q_\ell)\}_{\ell \in \naturals}$ satisfying
    \begin{align*}
     \lim_{\ell \to \infty} \frac{\log (1/q_\ell)  }{ \log N_\ell} =2\beta , \quad  \lim_{\ell \to \infty}  \frac{\log K_\ell}{ \log N_\ell}= \alpha
    \end{align*}
    such that for any sequence of randomized polynomial-time tests $\phi_{\ell}: \{0,1\}^{\binom{N_\ell}{2} } \to \{0,1\}$ for the $\PDS(N_\ell,K_\ell,2q_\ell,q_\ell)$ problem, the Type-I+II error probability is lower bounded by
    \begin{align*}
\liminf_{\ell \to \infty} \Prob_0\{ \phi_\ell(G')=1\}+ \Prob_1 \{ \phi_\ell(G')=0 \} \geq 1,
\end{align*}
where $G' \sim \calG(N,q)$ under $H_0$ and $G' \sim \calG(N,K,p,q)$ under $H_1$.
\end{theorem}

	\begin{proof}
	  Let $m_0=\lfloor \log_2 (1/\gamma) \rfloor$.  By \prettyref{eq:computationallimit}, there exist $0<\gamma \le 1/2$ and thus $m_0$ such that 
	     \begin{align}
    2\beta <\alpha <  \frac{1}{2} + \frac{m_0\beta+2}{2m_0\beta+1} \beta - \frac{1}{m_0 \beta}.   \label{eq:HardRegimeExpression}
    \end{align}
Fix $ \beta>0$ and $0<\alpha<1$ that satisfy \prettyref{eq:HardRegimeExpression}. 
Let $\delta=1/(m_0 \beta)$. Then it is straightforward to verify that 
$\frac{2+m_0\delta }{2+\delta} \beta \ge \frac{1}{2}-\delta + \frac{1+2\delta}{2+\delta} \beta $.
It follows from the assumption \prettyref{eq:HardRegimeExpression} that
    \begin{align}
    2\beta <\alpha< \min \left \{  \frac{2+m_0\delta }{2+\delta} \beta, \; \frac{1}{2}-\delta + \frac{1+2\delta}{2+\delta} \beta \right\}.
    \label{eq:betaub}   
    \end{align}
Let $\ell \in \naturals$ and $q_\ell=\ell^{-(2+\delta)}$.
Define
    \begin{align}
     n_\ell= \lfloor \ell^{\frac{2+\delta}{2\beta}-1 } \rfloor, \; k_\ell=\lfloor \ell^{\frac{(2+\delta)\alpha}{2\beta}-1 } \rfloor,\; N_\ell=n_\ell\ell, \; K_\ell=k_\ell\ell. \label{eq:DefNK}
    \end{align}
Then
\begin{align}
\lim_{\ell \to \infty} \frac{\log \frac{1}{q_\ell} }{ \log N_\ell} & = \frac{(2+\delta)}{(2+\delta)/(2\beta)-1 +1 } =2\beta, \nonumber \\
\lim_{\ell \to \infty}  \frac{\log K_\ell}{ \log N_\ell}  & = \frac{(2+\delta)\alpha/(2\beta)-1 +1 }{(2+\delta)/(2\beta)-1 +1} = \alpha. \label{eq:NKq}
\end{align}
Suppose that for the sake of contradiction there exists a small $\epsilon>0$ and a sequence of randomized polynomial-time tests $\{\phi_\ell\}$ for $\PDS(N_\ell,K_\ell,2q_\ell,q_\ell)$, such that
\begin{align*}
\Prob_0 \{ \phi_{N_\ell,K_\ell}(G') =1 \} + \Prob_1 \{ \phi_{N_\ell,K_\ell}(G')=0 \} \le 1-\epsilon
\end{align*}
holds for arbitrarily large $\ell$, where $G'$ is the graph in the $\PDS(N_\ell,K_\ell,2q_\ell,q_\ell)$.
Since $\alpha>2\beta$, we have $k_{\ell} \ge \ell^{1+\delta}$.
Therefore, $16 q_{\ell} \ell^2 \le 1$ and $k_{\ell} \ge 6e\ell$ for all sufficiently large $\ell$. Applying \prettyref{prop:reduction}, we conclude that $G \mapsto \phi(\tG)$ is a randomized polynomial-time test for $\PC(n_\ell,k_\ell, \gamma)$ whose Type-I+II error probability satisfies
\begin{align}
\Prob_{H_0^{\rm C}} \{  \phi_{\ell}( \tG ) =1\} + \Prob_{H_1^{\rm C} } \{\phi_\ell(\tG ) =0 \} \le 1- \epsilon+ \xi,
\label{eq:PlantedCliqueContradiction}
\end{align}
where $\xi$ is given by the right-hand side of \prettyref{eq:defxi}. By the definition of $q_\ell$, we have $q_\ell \ell^2 =\ell^{-\delta}$ and thus
\begin{align*}
k_\ell^2 (q_\ell\ell^2)^{m_0+1} \le \ell^{  (2+\delta)\alpha/\beta -2 - (m_0+1) \delta}  \le \ell^{-\delta},
\end{align*}
where the last inequality follows from \prettyref{eq:betaub}. Therefore $\xi\to0$ as $\ell\diverge$. Moreover, by the definition in \prettyref{eq:DefNK},
\begin{align*}
  \lim_{\ell \to \infty} \frac{\log k_\ell}{ \log n_\ell} = \frac{(2+\delta)\alpha/(2\beta)-1 }{(2+\delta)/(2\beta)-1}\le \frac{1}{2}- \delta,
\end{align*}
where the above inequality follows from \prettyref{eq:betaub}. Therefore, \prettyref{eq:PlantedCliqueContradiction} contradicts the assumption that \PC Hypothesis (\prettyref{hyp:HypothesisPlantedClique}) holds for $\gamma$.
\end{proof}

	\section{Discussions and open problems}
	\label{sec:discuss}
	
	Recent years have witnessed a great deal of progress on understanding the information-theoretical and
	computational limits of various statistical problems with planted structures. As outlined in this survey,
	various techniques are in place to identify the information-theoretic limits. In some cases, 
	polynomial-time procedures are shown to achieve the information-theoretic limits. However, in many
	other cases, it is believed that there exists a wide gap between the information-theoretic limits and
	the computational limits. For the planted clique problem,
	a recent exciting line of research has identified the performance limits of sum-of-squares hierarchy~\cite{Meka15,DeshpandeMontanari15,HKP15,RS15,barak-etal-planted-clique}. Under \PC Hypothesis, complexity-theoretic
	computational lower bounds have been derived for sparse PCA~\cite{berthet2013lowerSparsePCA}, 
	submatrix location~\cite{ma2013submatrix}, single community detection~\cite{HajekWuXu14}, and various other detection problems with
	planted structures~\cite{Brennan18}.
	 Despite these encouraging results, a variety of interesting questions remain open. Below we list a few representative problems.
	 Closing the observed computational gap, or equally importantly, disproving the possibility thereof on rigorous complexity-theoretic grounds, is an exciting new topic at the intersection of high-dimensional statistics, information theory, and computer science.

\paragraph{Computational lower bounds for recovering the planted dense subgraph} 
Closely related to the \PDS detection problem is the recovery problem, where given a graph generated from $\calG(N,K,p,q)$, the task is to recover the planted dense subgraph. Consider the asymptotic regime 
depicted in~\prettyref{fig:phase}. It has been shown in \cite{ChenXu14,ames2013robust} that exact recovery is 
information-theoretically possible if and only if $\beta<\alpha/2$ and can be achieved in polynomial-time if 
$\beta<\alpha -\frac{1}{2}$. Our computational lower bounds for the \PDS detection problem imply that 
the planted dense subgraph is hard to approximate to any constant factor if $ \max(0,2\alpha-1) <\beta<\alpha/2$ (the red regime in \prettyref{fig:phase}). Whether the planted dense subgraph is hard to approximate with any constant factor in the regime 
of $ \alpha -\frac{1}{2} \le \beta \le \min\{ 2\alpha-1, \alpha/2\}$ is an interesting open problem.
For the Gaussian case, \cite{CLR15} showed that exact recovery is computationally hard $\beta>\alpha -\frac{1}{2}$ by assuming 
a variant of the standard \PC hypothesis (see \cite[p.~1425]{CLR15}).

Finally, we note that to prove our computational lower bounds for the planted dense subgraph detection problem in \prettyref{thm:pds}, 
we have assumed the \PC detection problem is hard for any constant $\gamma>0$. 
An important open problem is to show by means of reduction that 
if \PC detection problem is hard with $\gamma=0.5$, then it is also hard with $\gamma=0.49$.

\paragraph{Computational lower bounds within the Sum-of-Squares Hierarchy} 
For the single community model, 
\cite{HajekWuXu_one_sdp15} obtained a tight characterization of the performance limits of SDP relaxations, corresponding to the sum-of-squares hierarchy with degree $2$. In particular, (1) if $K = \omega(n/ \log n)$, SDP attains the information-theoretic threshold with sharp constants; (2) If $K = \Theta(n/ \log n)$, SDP is suboptimal by a constant factor; (3) If $K = o(n/ \log n)$ and $K \to \infty$, SDP is order-wise suboptimal.
An interesting future direction to generalize this result to the sum-of-squares hierarchy, showing that sum-of-squares  with any constant degree are sub-optimal when $K=o(n\log n)$.

Furthermore, if $K \geq \frac{n}{\log n} (1/(8e) + o(1))$ for the Gaussian case and $K \geq \frac{n}{\log n} (\rho_{\sf BP}(p/q) + o(1))$
for the Bernoulli case,
exact recovery can be attained in nearly linear time via message passing plus clean up~\cite{HajekWuXu_one_beyond_spectral15,HajekWuXu_MP_submat15} whenever information-theoretically possible. An interesting question is whether exact recovery
beyond the aforementioned two limits is possible in polynomial-time.


\paragraph{Recovering multiple communities}
Consider the stochastic block model under which $n$ vertices are partitioned into $k$ equal-sized communities,
and two vertices are connected by an edge with probability $p$ if they are from the same community and
$q$ otherwise.

First let us focus on correlated recovery in the sparse regime where $p=a/n$ and $q=b/n$ for two fixed constants $a>b$ 
in the assortative case. 
For $k=2$, it has been shown~\cite{Mossel12,Massoulie13,Mossel13} that the information-theoretic and computational thresholds 
coincide at $(a-b)^2=2(a+b)$. Based on statistical physics heuristics, it is further conjectured that the information-theoretic and computational thresholds
continue to coincide for $k=3,4$, but depart from each other for $k\ge 5$; however, a rigorous proof remains open.

Next let us turn to exact recovery in the relatively sparse regime where $p=a \log n/n$ and $q=b \log n/n$ for two fixed
constants $a>b$. For $k=\Theta(1)$, it has been shown that the semidefinite programming (SDP) relaxations achieve the information-theoretic limits $\sqrt{a}-\sqrt{b}>\sqrt{k}$. Furthermore, it is shown that SDP continues to be optimal for $k=o(\log n)$, but cease to be optimal
for $k=\Theta(\log n)$. It is conjectured in~\cite{ChenXu14} that no polynomial-time procedure can be optimal for $k=\Theta(\log n)$.

\paragraph{Estimating graphons}
Graphon is a powerful network model for studying large networks \cite{Lovasz12}.
Concretely, given 
$n$ vertices, the edges are generated
independently, connecting each pair of two distinct
vertices $i$ and $j$ with a probability $ M_{ij} = f (x_i, x_j)$,
where $x_i \in [0,1]$ is the latent feature vector of vertex $i$ that
captures various characteristics of vertex $i$; $f: [0,1] \times [0,1] 
\to [0,1]$ is a symmetric function called graphon. 
The problem of interest is 
to estimate  either the edge probability matrix $M$ or the graphon $f$ on the basis of the observed graph. 
\begin{itemize}
	\item When $f$ is a step function which corresponds to the stochastic block model
with $k$ blocks for some $k$, the minimax optimal estimation error rate is 
shown to be on the order of $k^2/n^2 + \log k/n$~\cite{gao2015rate}, while
the currently best error rate achievable in polynomial-time is $k/n$~\cite{klopp2017optimal}.
\item When $f$ belongs to H\"{o}lder or Sobolev space with smoothness
 index $\alpha$, the minimax optimal rate is shown to be $n^{-2\alpha/(\alpha+1)}$ for $\alpha<1$ 
 and $\log n/n$ for $\alpha>1$~\cite{gao2015rate},
while 
 the best error rate achievable in polynomial-time that is known in the literature is
 $n^{-2\alpha/(2\alpha+1)}$~\cite{Xugraphon17}. 
\end{itemize}
For both cases, it remains open whether the minimax optimal rate can be achieved in polynomial-time.

\paragraph{Sparse PCA} 

Consider the following \emph{spiked Wigner} model, where the underlying signal is a rank-one matrix:
\begin{align} \label{eq:sparse_pca_def}
	X = \frac{\lambda}{\sqrt n} vv^\top + W \, ,
\end{align}
Here, $v \in \reals^n$, $\lambda > 0$ and $W \in \reals^{n \times n}$ is a Wigner random matrix with $W_{ii} \iiddistr \calN(0, 2)$ 
and $W_{ij}=W_{ij} \iiddistr \calN(0, 1)$ for $i<j$.  We assume for some $\gamma \in [0,1]$ the support of $v$ is drawn uniformly from all ${n \choose \gamma n}$ subsets $S \subset [n]$ with $|S| = \gamma n$.  Once the support is chosen, each nonzero component $v_i$ is drawn independently and uniformly from $\{\pm \gamma^{-1/2} \}$, so that 
$\|v\|_2^2 = n$.   When $\gamma$ is small, the data matrix $X$ is a sparse, rank-one matrix contaminated by Gaussian noise. 
For detection, we also consider a null model of $\lambda=0$ where $X=W$.

One natural approach for this problem is PCA: that is, diagonalize $X$ and use its leading eigenvector $\hat{v}$ as an estimate of $v$. 
Using the theory of random matrices with rank-one perturbations~\cite{baik2005phase,peche2006largest,benaych2011eigenvalues},
both detection and correlated recovery of $v$ is possible if and only if $\lambda > 1$. Intuitively, PCA only exploits the low-rank structure of the underlying signal, and not the sparsity of $v$; it is natural to ask whether one can succeed in detection or reconstruction for some $\lambda < 1$ by taking advantage of this additional structure. 
Through analysis of an approximate message-passing algorithm and the free energy, it is conjectured~\cite{lesieur2015phase,KrzakalaXuZdeborova16} that there exists a critical sparsity threshold $\gamma^\ast \in (0,1)$ such that if $\gamma \ge \gamma^\ast$, then  both the information-theoretic and computational thresholds are given by $\lambda = 1$; if 
$\gamma < \gamma^\ast$, then the computational threshold is given by $\lambda = 1$, but the information-theoretic threshold for $\lambda$ is strictly smaller.  A recent series of paper has identified the sharp information-theoretic threshold for correlated recovery through the Guerra interpolation technique and cavity method~\cite{KrzakalaXuZdeborova16,Barbier16,LelargeMiolane16,alaoui2018estimation}.
Also, the sharp information-theoretic threshold for detection has been recently determined  in~\cite{alaoui2017finite}. 
However, there is no rigorous evidence  
justifying that $\lambda=1$ is the computational threshold.

\paragraph{Tensor PCA}
We can also consider a planted tensor model, in which we observe an order-$k$ tensor
\begin{align}
	X = \lambda  v^{\otimes k}  + W  
	\label{eq:planted_tensor_model}
\end{align}
where $v$ is uniformly distributed over the unit sphere in $\reals^n$ 
and $W \in (\reals^{n})^{\otimes k}$ is a totally symmetric noise tensor with Gaussian entries $\calN(0,1/n)$ (see \cite[Section 3.1]{MontanariReichmanZeitouni14} for a precise definition).  
This model is known as the \emph{$p$-spin model} in statistical physics, and is widely used in machine learning and data analysis to model high-order correlations in a dataset. A natural approach is tensor PCA, which coincides with the maximum likelihood estimator: $\min_{\|u\|_2 = 1} \iprod{X}{u^{\otimes k}}$.
When $k=2$, this reduces to standard PCA which can be efficiently computed by singular value decomposition; however, as soon as $k \ge 3$, tensor PCA becomes NP-hard in the worst case~\cite{Hillar2013}. 

Previous work~\cite{MontanariRichard14,MontanariReichmanZeitouni14,PerryWeinBandeira16}
shows that tensor PCA achieves consistent estimation of $v$ if $\lambda \gtrsim \sqrt{k \log k}$, while these are information-theoretically impossible if $\lambda \lesssim \sqrt{k \log k}$.  
The exact location of the information-theoretic threshold for any $k$ was determined recently in~\cite{LMLKZ17}, but all known polynomial-time algorithms fail far from this threshold.  A ``tensor unfolding'' algorithm is shown in~\cite{MontanariRichard14} to succeed if $\lambda \gtrsim n^{(\lceil k/2 \rceil -1 )/2}$.  In the special case $k=3$, it is further shown in~\cite{hopkins2015tensor} that a degree-$4$ sum-of-squares relaxation succeeds if $\lambda = \omega(n\log n)^{1/4}$ and fails if $\lambda=O(n/\log n)^{1/4}$. More recent work~\cite{zhang2017tensor} shows that a spectral method achieves consistent estimation provided that 
$\lambda=\Omega(n^{1/4})$, improving the positive result in \cite{hopkins2015tensor}  by a poly-logarithmic factor. 
It remains open whether any polynomial-time algorithm succeeds in the regime of $1 \lesssim \lambda \lesssim n^{1/4}$. 
Under a hypergraph version of the planted clique detection hypothesis, it is shown in ~\cite{zhang2017tensor} that  
 no polynomial-time algorithm can succeed when $\lambda \le n^{1/4-\epsilon}$ for an arbitrarily small constant $\epsilon>0$. 
It remains open whether the usual planted clique problem can be reduced to the hypergraph version.


\paragraph*{Gaussian mixture clustering}
Consider the following model of clustering in high dimensions. 
Let $v_1,...,v_k$ be independently and identically distributed as $\calN \left(0, k/(k-1) \, \identity_{n} \right)$,
 and define $\overline v = (1/k)\sum_{s} v_s$ to be their mean. The scaling of the expected norm of each $v_s$ with $k$ ensures that $\Exp \|v_s - \overline{v} \|_2^2 =n$ for all $1 \le s \le k$. For a fixed parameter $\alpha>0$, we then generate $m = \alpha n$ points $x_i \in \reals^n$ which are partitioned into $k$ clusters of equal size by a balanced partition $\sigma:[n]\to[k]$, again chosen uniformly at random from all such partitions. For each data point $i$, let $\sigma_i \in [k]$ denote its cluster index, and generate $x_i$ independently according to Gaussian distribution with mean $\sqrt{\rho/n}\, (v_{\sigma_i} - \overline v)$ and identity covariance matrix, where  $\rho>0$ is a fixed parameter characterizing the separation between clusters. 
Equivalently, this model can be described by in the following matrix form:
\begin{align} \label{eq:clustering_def}
	X = \sqrt{\frac{\rho}{ n}} \left(S - \frac{1}{k}\allones_{m,k}\right) V^\top + W,
\end{align}
where $X=[x_1,\ldots,x_m]^\top$, $V = [v_1, ..., v_k ]$, $S$ is $m \times k$ with $S_{i,t} = \indicator{\sigma_i = t}$, 
$\allones_{m,k}$ is the $m\times k$ all-one matrix,  and $W_{i,j} \iiddistr \calN(0,1)$.
  In the null model, there is no cluster structure and $X = W$. 
The subtraction of $\allones_{m,k}/k$ centers the signal matrix so that $\Exp X = 0$ in both models. 
It follows from the celebrated BBP phase transition~\cite{baik2005phase,Paul07} that detection
and correlated recovery using spectral methods is possible if and only if $\rho \sqrt{\alpha} >(k-1) $. 
In contrast, detection and correlated recovery are shown to be information-theoretically 
possible if $\rho > 2\sqrt{ \frac{k \log k}{\alpha} } + 2\log k $. The sharp characterization of
information-theoretic limit is still open and it is conjectured~\cite{clustering} that the computational threshold
coincides with the spectral detection threshold.





\begin{appendices}

\section{Mutual information-characterization of correlated recovery}
\label{app:MIcorr}

\newcommand{\Unif}{\mathrm{Unif}}

We consider a general setup:  Let the number of communities $k$ be a constant. Denote the membership vector by $\sigma=(\sigma_1,\ldots,\sigma_n) \in [k]^n$ and the observation is $A=(A_{ij}: 1 \leq i < j \leq n)$. Assume the following conditions:
\begin{enumerate}[label=A\arabic*]
	\item \label{A1}
	For any permutation $\pi\in S_k$, $(\sigma,A)$ and $(\pi(\sigma),A)$ are equal in law, where $\pi(\sigma)\triangleq (\pi(\sigma_1),\ldots,\pi(\sigma_n))$; 
	
	\item \label{A2}
	For any $i \neq j \in [n]$, $I(\sigma_i,\sigma_j;A)= I(\sigma_1,\sigma_2;A)$;
	
	\item \label{A3}
	For any $z_1,z_2 \in [k]$, $\prob{\sigma_1=z_1,\sigma_2=z_2} = \frac{1}{k^2} + o(1)$ as $n\to \infty$.
\end{enumerate}
These assumptions are satisfied for example for $k$-community SBM (where each pair of vertices $i$ and $j$ are connected independently with probability $p$ if $\sigma_i=\sigma_j$ and $q$ otherwise),\index{Stochastic block model (SBM)! $k$ communities}
 and the membership vector
$\sigma$ can be either uniformly distributed on $[k]^n$ or the set of equal-sized $k$-partition of $[n]$. 

Recall that correlated recovery entails the following:
For any $\sigma , \hat{\sigma} \in [k]^n$, define the overlap:
\begin{align}
o\left( \sigma, \hat{\sigma} \right) = \frac{1}{n} \max_{\pi \in S_k} 
\sum_{i \in [n] } \left( \indc{ \pi\left(\sigma_i \right) = \hat{\sigma}_i} - \frac{1}{k} \right).
\end{align}
We say an estimator $\hat{\sigma} = \hat{\sigma}(A)$ achieves correlated recovery if\footnote{For the special case of $k=2$, \prettyref{eq:corro} is equivalent to 
$\frac{1}{n}\Expect[|\Iprod{\sigma}{\hat \sigma}|] = \Omega(1)$, where $\sigma , \hat{\sigma}$ are assumed to be $\{\pm\}^n$-valued.}
\begin{equation}
\expect{o\left( \sigma, \hat{\sigma} \right)}=\Omega(1),
\label{eq:corro}
\end{equation}
 that is, the misclassification rate, up to a global permutation, outperforms random guessing.
Under the above three assumptions, we have the following characterization of correlated recovery:
\begin{lemma}
\label{lmm:MIcorr}	
Correlated recovery is possible if and only if 
$I(\sigma_1, \sigma_2 ; A) = \Omega(1)$.	
\end{lemma}

\begin{proof}
We start by recalling the relation between mutual information and total variation.
For any pair of random variables $(X,Y)$, define the so-called $T$-information \cite{Csiszar96}:
$T(X;Y) \triangleq \TV(P_{XY}, P_XP_Y) = \Expect[\TV(P_{Y|X}, P_Y)]$.
For $X \sim \Bern(p)$, this simply reduces to 
\begin{equation}
T(X;Y) = 2p(1-p) \TV(P_{Y|X=0}, P_{Y|X=1}).
\label{eq:Tbern}
\end{equation}
Furthermore, the mutual information can be bounded by the $T$-information, 
by Pinsker's and Fano's inequality, as follows \cite[Eq.~(84) and Prop.~12]{PW14a}
\begin{equation}
 2 T(X;Y)^2 \leq I(X;Y)  \leq \log (M-1) T(X;Y) + h(T(X;Y))		
	\label{eq:TI}
\end{equation}
where in the upper bound $M$ is the number of possible values of $X$, and $h$ is the binary entropy function in \prettyref{eq:binaryentropy}.

We prove the ``if'' part.
Suppose 
$I(\sigma_1, \sigma_2; A) = \Omega(1)$.
We first claim that assumption \ref{A1} implies that 
\begin{equation}
I (\indc{\sigma_1 =\sigma_2}; A)=I(\sigma_1, \sigma_2; A)
\label{eq:ss}
\end{equation}
that is,
 $A$ is independent of $\sigma_1,\sigma_2$ conditional on $\indc{\sigma_1 =\sigma_2}$. 
Indeed, 
for any $z \neq z'\in [k]$, let $\pi$ be any permutation such that 
$\pi(z')=z.$
Since $P_{\sigma, A} =  P_{\pi(\sigma), A}$, 
we have $P_{A|\sigma_1=z,  \sigma_2=z} =  P_{A| \pi(\sigma_1)=z,  \pi(\sigma_2)=z }$, i.e., 
$P_{A |\sigma_1=z, \sigma_2=z} =  P_{A| \sigma_1 =z',  \sigma_2 =z' }$. 
Similarly, one can show that 
$P_{A|\sigma_1=z_1, \sigma_2=z_2} =  P_{A| \sigma_1 =z_1',  \sigma_2 =z_2'}$, 
for any $z_1 \neq z_2$ and $z_1'\neq z_2'$, and this proves the claim.

Let $x_j=\indc{\sigma_1 =\sigma_j}$.
By the symmetry assumption \ref{A2}, 
$I(x_j; A) = I(x_2; A) = \Omega(1)$ for all $j \neq 1$.
Since $\prob{x_j = 1} = \frac{1}{k} + o(1)$ by assumption \ref{A3}, applying \prettyref{eq:TI} with $M=2$ and in view of \prettyref{eq:Tbern}, we have
$\TV(P_{A|x_j=0},P_{A|x_j=1})=\Omega(1)$.
Thus, there exists an estimator $\widehat{x}_j \in \{0,1\}$ as a function of $A$, such that
\begin{align}
\prob{\widehat{x}_j = 1 \mid x_j =1 } + 
\prob{\widehat{x}_j = 0 \mid x_j =0 } \ge 1+\TV(P_{A|x_j=0},P_{A|x_j=1})
= 1+ \Omega(1).  \label{eq:estimator_x_hat}
\end{align}

Define $\hat\sigma$ as follows: set $\hat\sigma_1= 1 $; for $j \neq 1$, set $\hat \sigma_j = 1 $ if $\widehat{x}_j = 1$ 
and draw $\hat \sigma_j $ from $\{2,\ldots,k\}$ uniformly at random 
if $\widehat{x}_j = 0$.
Next, we show that $\hat\sigma$ achieves correlated recovery. 
Indeed, fix a permutation $\pi \in S_k$ such that $\pi(\sigma_1)=1$. It follows from 
the definition of overlap that
\begin{equation}
\Expect[o\left(\sigma, \hat{\sigma} \right)]
\ge \frac{1}{n} \sum_{j \neq 2} \prob{\pi(\sigma_j) =\hat\sigma_j} - \frac{1}{k}.
\label{eq:olb}
\end{equation}
Furthermore, since $\pi(\sigma_1)=1$, we have, for any $j\neq 1$,
\[
\prob{\pi(\sigma_j) = \hat\sigma_j,x_j=1}=
\prob{\hat x_j=1,x_j=1}
\]
and
\[
\prob{\pi(\sigma_j) = \hat\sigma_j,x_j=0}=
\prob{\pi(\sigma_j) = \hat\sigma_j, \hat x_j=0,x_j=0}=
\frac{1}{k-1} \prob{\hat x_j=0,x_j=0},
\]
where the last step is because conditional on $\hat{x}_j=0$,
$\hat{\sigma}_j$ is chosen  from $\{2,\ldots,k\}$ uniformly and 
independently of everything else.
Since $\prob{x_j = 1} = \frac{1}{k} + o(1)$, we have
\[
\prob{\pi(\sigma_j) =\hat\sigma_j} = \frac{1}{k}(\prob{\widehat{x}_j = 1 \mid x_j =1 } + 
\prob{\widehat{x}_j = 0 \mid x_j =0 }) +o(1) \overset{\prettyref{eq:estimator_x_hat}}{\ge} \frac{1}{k}+ \Omega(1).  
\]
By \prettyref{eq:olb}, we conclude that $\hat{\sigma}$ achieves correlated recovery
of $\sigma$.

Next we prove the ``only if'' part.
Suppose $I(\sigma_1, \sigma_2; A) = o(1)$ and we aim to show 
$\expect{o\left(\sigma, \hat{\sigma} \right)}=o(1)$ for any estimator $\hat\sigma$.
By the definition of overlap, we have
\begin{align*}
o\left(\sigma, \hat{\sigma} \right) 
\le 
\frac{1}{n} \sum_{\pi \in S_k}
\left| \sum_{i \in [n] }  
\left( \indc{ \pi\left(\sigma_i \right) = \hat{\sigma}_i} - \frac{1}{k}  \right) \right|.
\end{align*}
Since there are $k!=\Omega(1)$ permutations in $S_k$, it suffices to show
for any fixed permutation $\pi$,
$$
\expect{ \left| \sum_{i \in [n] } \left( \indc{ \pi\left(\sigma_i \right) = \hat{\sigma}_i} - \frac{1}{k}  \right) 
\right| } = o(n).
$$
Since $I(\pi(\sigma_i), \pi(\sigma_j); A)=I(\sigma_i, \sigma_j; A)$, without loss of generality, 
we assume $\pi=\text{id}$ in the following. By the Cauchy-Schwarz inequality, it further suffices to show
\begin{equation}
\expect{ \left( \sum_{i \in [n] } \left( \indc{ \sigma_i =\hat{\sigma}_i } - \frac{1}{k} \right)\right)^2 } =o(n^2).
\label{eq:overlap2}
\end{equation}
Note that 
\begin{align*}
& \expect{  \left( \sum_{i \in [n] } \left(
   \indc{ \sigma_i  = \hat{\sigma}_i } - \frac{1}{k} \right)   \right)^2 } \\
 & = \sum_{i, j \in [n] } 
 \expect{ \left( \indc{ \sigma_i = \hat{\sigma}_i }- \frac{1}{k}  \right) 
 \left( \indc{ \sigma_j  = \hat{\sigma}_j }- \frac{1}{k}  \right) }  \\
 & =  \sum_{i, j \in [n] }  \prob{  \sigma_i = \hat{\sigma}_i , \sigma_j  = \hat{\sigma}_j }
 - \frac{2n}{k} \sum_{i \in [n] }\prob{ \sigma_i = \hat{\sigma}_i } + \frac{n^2}{k^2}.
\end{align*}
For the first term in the last displayed equation, 
let $\sigma'$ be identically distributed as $\hat\sigma$ but independent of $\sigma$.
Since $I(\sigma_i,\sigma_j;\hat\sigma_i,\hat\sigma_j) \le I(\sigma_i,\sigma_j;A)=o(1)$ by the data processing inequality, it follows from the lower bound in 
\prettyref{eq:TI} that $\TV(P_{\sigma_i,\sigma_j,\hat\sigma_i,\hat\sigma_j}, P_{\sigma_i,\sigma_j,\sigma_i',\sigma_j'})=o(1)$.
Since 
$\pprob{  \sigma_i = {\sigma}_i', \sigma_j  = {\sigma}_j' } \leq 
\max_{a,b \in [k]} \prob{  \sigma_i = a, \sigma_j  = b} \leq \frac{1}{k^2}+o(1)$ by assumption \ref{A3}, 
we have
$$
\prob{  \sigma_i = \hat{\sigma}_i , \sigma_j  = \hat{\sigma}_j }
\leq \frac{1}{k^2} + o(1),
$$
Similarly, for the second term, we have
$$
\prob{ \sigma_i = \hat{\sigma}_i } = \frac{1}{k} +o(1),
$$
where the last equality holds due to $I(\sigma_i; A) =o(1).$
Combining the last three displayed equations gives \prettyref{eq:overlap2} and completes the proof.
\end{proof}

\section{Proof of \prettyref{eq:second_moment_conditional} $\implies$ \prettyref{eq:MIcorr2} and 
verification of  \prettyref{eq:second_moment_conditional} in the binary symmetric SBM}
\label{app:MITV}
Combining \prettyref{eq:ss} with 
\prettyref{eq:TI} and \prettyref{eq:Tbern}, we have
$I(\sigma_1,\sigma_2; A) = o(1)$ if and only if $\TV(\P_+,\P_-) =o(1)$,
where $\P_+=P_{A|\sigma_1=\sigma_2}$ and $\P_-=P_{A|\sigma_1\neq\sigma_2}$.
Note the following characterization about the total variation distance, which simply follows from the Cauchy-Schwartz inequality:
\begin{equation}
\TV(\P_+,\P_-) = \frac{1}{2} \sqrt{\inf_{\Q}  \int  \frac{  \left( \P_{+} - \P_{-} \right)^2 }{\Q }}
\label{eq:TVquadratic}
\end{equation}
where the infimum is taken over all probability distributions $\Q$. 
Therefore \prettyref{eq:second_moment_conditional} implies \prettyref{eq:MIcorr2}.

%

Finally, we consider the binary symmetric SBM and show that,
below the correlated recovery threshold $\tau=\frac{(a-b)^2}{2(a+b)}<1$, 
 \prettyref{eq:second_moment_conditional} is satisfied if the reference distribution $\Q$ is the distribution of $A$ in
the null (\ER) model. Note that 
$$
\int  \frac{  \left( \P_{+} - \P_{-} \right)^2 }{\Q }  =
\int \frac{\P^2_{+}}{\Q} +\int \frac{\P^2_{-}}{\Q} - 2 \int \frac{ \P_{+} \P_{-} }{ \Q}.
$$
Hence, it is sufficient to show
$$
 \int \frac{ \P_{z} \P_{\tilde{z} } }{ \Q} =  
C+o(1), \quad  \forall z, \tilde{z} \in \{\pm \}
$$
for some constant $C$ independent of $z$ and $\tz$.
Specifically, following the derivations in \prettyref{eq:SBM_second_moment_eq},
we have
\begin{align}
\int \frac{   \P_{z}  \P_{\tz} }{ \Q } 
&= \Expect \qth{  \prod_{i < j} \left(1 +  \sigma_i \sigma_j \tsigma_i \tsigma_j \rho \right)
\mathrel{\bigg|} \sigma_1 \sigma_2=z, \tsigma_1 \tsigma_2 =\tz }  \nonumber \\
& =  \left( 1+o(1) \right) e^{ -\tau^2/4 -\tau/2} \times 
\Expect \qth{ \exp \left( \frac{\rho}{2} \iprod{ \sigma}{\tsigma}^2  \right) \mathrel{\Big|} \sigma_1 \sigma_2=z, \tsigma_1 \tsigma_2 =\tz },
\end{align}
where the last equality holds $\rho=\tau/n + O(1/n^2)$ and $\log(1+x) = x -x^2/2 +O(x^3)$. 

Write $\sigma=2\xi-1$ for $\xi\in\{0,1\}^n$ and let 
$$
H_1\triangleq \xi_1 \tilde{\xi}_1 + \xi_2 \tilde{\xi}_2 
\quad \text{ and } \quad 
H_2 \triangleq \sum_{j \ge 3}^n \xi_j \tilde{\xi}_j.
$$ 
Then $\iprod{\sigma}{\tsigma} = 4 (H_1+H_2) -n$.
Moreover, conditional  on $\sigma_1, \sigma_2$ and $\tsigma_1, \tsigma_2$, 
$$
H_2 \sim \text{Hypergeometric} \left( n-2, n/2 - \xi_1-\xi_2, n/2 - \tilde{\xi}_1-\tilde{\xi}_2 \right).
$$
Since $|H_1| \le 2$, $\xi_1+\xi_2 \le 2 $, and $\tilde{\xi}_1+\tilde{\xi}_2 \le 2$, 
it follows that 
conditional on $\sigma_1 \sigma_2=z, \tsigma_1 \tsigma_2 =\tz $,
$ \frac{1}{\sqrt{n}} ( 4H_1 + 4H_2 - n )$ converges to $\calN(0,1)$ in distribution 
as $n \to \infty$ by the central limit theorem for hypergeometric distribution. 
Therefore
\begin{align*}
& \Expect \qth{ \exp \left(  \frac{\rho}{2} \iprod{ \sigma}{\tsigma}^2  \right) \mathrel{\bigg|} \sigma_S=z, \tsigma_S =\tz} \\
& =\expect{ \exp \left(  \frac{n\rho}{2} \left( \frac{ 4 H_1 + 4H_2 - n}{\sqrt{n} } \right)^2 \right) \mathrel{\bigg|} 
\sigma_1 \sigma_2=z, \tsigma_1 \tsigma_2 =\tz }  \\
& = \frac{1+o(1)}{\sqrt{1-\tau} },
\end{align*}
where the last equality holds due to $n \rho = \tau +o(1/n)$, $\tau<1$, 
and the convergence of the moment generating function. 


\end{appendices}

 \newcommand{\etalchar}[1]{$^{#1}$}


\end{document}